\RequirePackage[l2tabu,orthodox]{nag}		
\documentclass[reqno]{amsart}

\usepackage{amsmath}
\usepackage{amssymb}
\usepackage{amsfonts}
\usepackage{amsthm}
\usepackage[foot]{amsaddr}

\usepackage{mathtools}
\mathtoolsset{%
}

\usepackage[utf8]{inputenc}		
\usepackage[T1]{fontenc}		

\usepackage[sf,mono=false]{libertine}

\usepackage[sans]{dsfont}		

\usepackage[scale=0.94,ttdefault=true]{AnonymousPro}

\usepackage[
cal=cm,
]
{mathalfa}

\usepackage[labelfont={bf,small},labelsep=colon,font=small]{caption}	
\usepackage[dvipsnames,svgnames]{xcolor}		
\colorlet{MyBlue}{DodgerBlue!75!Black}
\colorlet{MyGreen}{DarkGreen!85!Black}

\usepackage{wasysym}		

\usepackage{subcaption}		
\usepackage{tikz}		
\usetikzlibrary{calc,patterns}

\usepackage{array}		
\usepackage{booktabs}		
\usepackage{paralist}		

\usepackage{microtype}		

\usepackage{acronym}		
\usepackage{latexsym}		
\usepackage{xspace}		

\usepackage{cite}

\usepackage{hyperref}
\hypersetup{
final,
colorlinks=true,
linktocpage=true,
pdfstartview=FitH,
breaklinks=true,
pdfpagemode=UseNone,
pageanchor=true,
pdfpagemode=UseOutlines,
plainpages=false,
bookmarksopen=false,
hypertexnames=true,
pdfhighlight=/O,
urlcolor=MyBlue!60!black,linkcolor=MyBlue!70!black,citecolor=DarkGreen!70!black, 
pdftitle={},
pdfauthor={},
pdfsubject={},
pdfkeywords={},
pdfcreator={pdfLaTeX},
pdfproducer={LaTeX with hyperref}
}
\newcommand{\EMAIL}[1]{\email{\href{mailto:#1}{#1}}}

\numberwithin{equation}{section}  
\usepackage[sort&compress,capitalize,nameinlink]{cleveref}

\crefrangeformat{equation}{\upshape(#3#1#4)\textendash(#5#2#6)}

\newcommand{\itemref}[1]{\textup{(\ref{#1})}}

\usepackage[textwidth=30mm]{todonotes}		

\newcommand{\debug}[1]{#1}		
\newcommand{\revise}[1]{#1}		
\newcommand{\typo}[1]{#1}							

\usepackage{algorithm}		
\usepackage{algpseudocode}		

\theoremstyle{plain}
\newtheorem{theorem}{Theorem}

\newtheorem*{corollary*}{Corollary}
\newtheorem{lemma}[theorem]{Lemma}

\theoremstyle{definition}
\newtheorem{definition}[theorem]{Definition}
\newtheorem*{definition*}{Definition}
\newtheorem{assumption}{Assumption}

\theoremstyle{remark}
\newtheorem{remark}{Remark}
\newtheorem*{remark*}{Remark}
\newtheorem*{notation*}{Notational remark}
\newtheorem{example}{Example}

\renewcommand{\qed}{\hfill{\footnotesize$\blacksquare$}}

\newcounter{proofstep}

\numberwithin{theorem}{section}
\numberwithin{example}{section}

\DeclarePairedDelimiter{\bracks}{[}{]}		
\DeclarePairedDelimiter{\parens}{(}{)}		

\DeclarePairedDelimiter{\abs}{\lvert}{\rvert}		

\DeclarePairedDelimiterX{\setdef}[2]{\{}{\}}{#1:#2}		
\DeclarePairedDelimiterXPP{\exclude}[1]{\mathopen{}\setminus}{\{}{\}}{}{#1}

\newcommand{\cf}{cf.\xspace}		
\newcommand{\eg}{e.g.,\xspace}		
\newcommand{\ie}{i.e.,\xspace}		
\newcommand{\textpar}[1]{\textup(#1\textup)}		

\newcommand{\txs}{\textstyle}		

\newcommand{\N}{\mathbb{N}}		
\newcommand{\R}{\mathbb{R}}		

\DeclareMathOperator{\bigoh}{\mathcal O}		
\DeclareMathOperator{\grad}{\nabla}		
\DeclareMathOperator{\Hess}{\debug \nabla^{2}}		

\newcommand{\eps}{\varepsilon}		
\newcommand{\pd}{\partial}		

\newcommand{\vecspace}{\mathcal{\debug V}}		

\newcommand{\bvec}{\debug e}		

\DeclarePairedDelimiterX{\braket}[2]{\langle}{\rangle}{#1,#2}		

\DeclareMathOperator{\diag}{diag}		
\DeclareMathOperator{\image}{im}

\DeclareMathOperator{\relint}{ri}		

\newcommand{\tvec}{\debug z}		

\DeclareMathOperator*{\argmin}{arg\,min}		

\newcommand{\points}{\mathcal{\debug X}}		
\newcommand{\point}{\debug x}		
\newcommand{\pointalt}{\point'}		

\newcommand{\obj}{\debug f}		

\newcommand{\vecfield}{\debug v}		

\newcommand{\sol}[1][\point]{#1^{\ast}}		
\newcommand{\sols}{\sol[\points]}		

\newcommand{\cost}{\debug c}		

\newcommand{\hreg}{\debug h}		
\newcommand{\breg}{\debug D}		

\DeclareMathOperator{\ex}{\mathbb{E}}		
\DeclareMathOperator{\prob}{\mathbb{P}}		

\DeclarePairedDelimiterXPP{\exof}[1]{\ex}{[}{]}{}{
 #1}

\DeclarePairedDelimiterXPP{\probof}[1]{\prob}{(}{)}{}{
 #1}

\newcommand{\start}{\debug 0}		
\newcommand{\running}{\debug 0,1,\dotsc}		
\newcommand{\run}{\debug k}		
\newcommand{\runalt}{\debug r}		

\newcommand{\new}[1]{#1^{+}}		

\newcommand{\step}{\debug \alpha}		

\newcommand{\from}{\colon}		

\newcommand{\graph}{\mathcal{\debug G}}
\newcommand{\vertices}{\mathcal{\debug V}}
\newcommand{\edges}{\mathcal{\debug E}}

\DeclareMathOperator{\bd}{bd}		
\DeclareMathOperator{\dist}{dist}		

\newcommand{\open}{\mathcal{\debug U}}		

\DeclarePairedDelimiterX{\product}[2]{\langle}{\rangle}{#1,#2}		
\DeclarePairedDelimiter{\norm}{\lVert}{\rVert}		
\DeclarePairedDelimiterXPP{\dnorm}[1]{}{\lVert}{\rVert}{_{\ast}}{#1}		

\newcommand{\gmat}{\debug H}		

\newcommand{\Id}{\debug I}

\newcommand{\edge}{e}

\newcommand{\source}{o}

\newcommand{\sink}{d}

\newcommand{\pair}{i}

\newcommand{\nPairs}{N}
\newcommand{\pairs}{\mathcal{N}}

\newcommand{\rate}{m}

\newcommand{\flow}{x}
\newcommand{\flows}{\mathcal{X}}

\newcommand{\load}{w}

\newcommand{\Cost}{C}

\newcommand{\nRoutes}{P}
\newcommand{\routes}{\mathcal{\nRoutes}}
\newcommand{\route}{p}

\newcommand{\Rn}{\R^{\nDims}}
\newcommand{\Rnp}{\R_{+}^{\nDims}}

\newcommand{\scrA}{\mathcal{A}}
\newcommand{\scrB}{\mathcal{B}}

\newcommand{\scrK}{\mathcal{K}}
\newcommand{\scrL}{\mathcal{L}}

\newcommand{\scrP}{\mathcal{P}}

\newcommand{\scrS}{\mathcal{S}}

\newcommand{\clorthant}{\mathcal{\debug C}}
\newcommand{\orthant}{\revise{\relint(\clorthant)}}
\newcommand{\intpoints}{\revise{\relint(\points)}}
\newcommand{\nDims}{\debug n}
\newcommand{\nConsts}{\debug m}

\newcommand{\cmat}{\debug A}		
\newcommand{\cvec}{\debug b}		

\newcommand{\projmat}{\debug P}		
\newcommand{\hmat}{\debug H}		

\newcommand{\test}{\debug\phi}

\newcommand{\dsol}{\debug y^{\ast}}

\newcommand{\pexp}{\debug p}

\newcommand{\armijo}{\debug \mu}
\newcommand{\shrink}{\debug \delta}

\begin{document}


\newcommand{\acli}[1]{\textit{\acl{#1}}}
\newcommand{\acdef}[1]{\textit{\acl{#1}} \textup{(\acs{#1})}\acused{#1}}
\newcommand{\acdefp}[1]{\emph{\aclp{#1}} \textup(\acsp{#1}\textup)\acused{#1}}

\newacro{MGF}{metric generating function}
\newacro{DGF}{distance generating function}
\newacro{lsc}[l.s.c.]{lower semi-continuous}
\newacro{KKT}{Karush-Kuhn-Tucker}
\newacro{RN}{regularized Newton}
\newacro{AS}{affine scaling}
\newacro{RD}{replicator dynamics}
\newacro{LV}{Lotka-Volterra}
\newacro{ODE}{ordinary differential equation}
\newacro{HR}{Hessian Riemannian}
\newacro{HRGD}{Hessian Riemannian gradient descent}
\newacro{OD}[O/D]{origin-destination}
\newacro{TAP}{traffic assignment problem}
\newacro{MD}{mirror descent}
\newacro{HBA}{Hessian barrier algorithm}
\newacroplural{NE}[NE]{Nash equilibria}
\newacro{VI}{variational inequality}
\newacroplural{VI}[VIs]{variational inequalities}
\newacro{iid}[i.i.d.]{independent and identically distributed}

\title
[Hessian barrier algorithms]
{Hessian barrier algorithms for\\
linearly constrained optimization problems\footnotemark[2]}

\author[I.M.~Bomze]{Immanuel M. Bomze$^{\sharp}$}		
\address{$^{\sharp}$\,%
Universit\"{a}t Wien, ISOR/VCOR \& DS:UniVie, Vienna, Austria}		
\EMAIL{immanuel.bomze@univie.ac.at}		
\author[P.~Mertikopoulos]{Panayotis Mertikopoulos$^{\star}$}		
\address{$^{\star}$\,%
Univ. Grenoble Alpes, CNRS, Inria, Grenoble INP, LIG, Grenoble, France.}		
\EMAIL{panayotis.mertikopoulos@imag.fr}
\author[W.~Schachinger]{\\Werner Schachinger$^{\sharp}$}		
\EMAIL{werner.schachinger@univie.ac.at}		
\author[M.~Staudigl]{Mathias Staudigl$^{\diamond}$}		
\address{$^{\diamond}$\,%
Maastricht University, Department of Quantitative Economics, P.O. Box 616, NL\textendash 6200 MD Maastricht, The Netherlands.}		
\EMAIL{m.staudigl@maastrichtuniversity.nl}		


\keywords{%
\acl{HRGD};
interior-point methods;
\acl{MD};
non-convex optimization;
traffic assignment.
\vspace{1ex}}

\maketitle
\begin{abstract}
%
%
In this paper, we propose an interior-point method for linearly constrained \textendash\ and possibly nonconvex \textendash\ optimization problems.
The proposed method \textendash\ which we call the \acdef{HBA} \textendash\ combines a forward Euler discretization of \acl{HR} gradient flows with an Armijo backtracking step-size policy.
In this way, \ac{HBA} can be seen as an \revise{alternative} to \acf{MD}, and contains as special cases the \acl{AS} algorithm, \acl{RN} processes, and several other iterative solution methods.
Our main result is that, modulo a non-degeneracy condition, the algorithm converges to the problem's critical set;
hence, in the convex case, the algorithm converges globally to the problem's minimum set.
In the case of linearly constrained quadratic programs (not necessarily convex), we also show that the method's convergence rate is $\bigoh(1/k^{\rho})$ for some $\rho\in(0,1]$ that depends only on the choice of kernel function (\ie not on the problem's primitives).
These theoretical results are validated by numerical experiments in standard non-convex test functions and large-scale \aclp{TAP}.
\end{abstract}

\acresetall
\allowdisplaybreaks

\section{Introduction}
\label{sec:introduction}

Consider a linearly constrained optimization problem of the form
\begin{equation}
\label{eq:Opt}
\tag{Opt}
\begin{aligned}
\textrm{minimize}
	&\quad
	\obj(\point)
	\\
\textrm{subject to}
	&\quad
	\cmat\point = \cvec,\;
			\point\geq0.
\end{aligned}
\end{equation}
In this formulation, the primitives of \eqref{eq:Opt} are:%
\footnote{Inequality constraints of the form $\cmat\point\leq \cvec$ can also be accommodated in \eqref{eq:Opt} by introducing the corresponding slack variables $s = \cvec - \cmat\point \geq 0$.
Despite the slight loss in parsimony, the equality form of \eqref{eq:Opt} turns out to be more convenient in terms of notational overhead, so we stick with the standard equality formulation throughout.}
\begin{enumerate}
[\indent\itshape i\hspace*{.5pt}\upshape)]
\addtolength{\itemsep}{\smallskipamount}
\item
The problem's \emph{objective function} $\obj\from\clorthant\to\R\cup\{+\infty\}$, where $\clorthant \equiv \Rnp$ denotes the non-negative orthant of $\Rn$.
\item
The problem's \revise{\emph{feasible region}}
\begin{equation}
\label{eq:feas}
\revise{%
\points
	= \setdef{\point\in\Rn}{\cmat\point=\cvec, \point\geq0}
	}
\end{equation}
where $\cmat\in\R^{\nConsts\times\nDims}$ is a matrix of rank $\nConsts\geq0$ and $\cvec\in\R^{\nConsts}$ is an $\nConsts$-dimensional real vector (both assumed known to the optimizer).
\end{enumerate}
\smallskip

Problems of this type are ubiquitous:
they arise naturally in
data science and machine learning \cite{Bub15,HaeLiuYe18},
game theory and operations research \cite{Bom02,NRTV07},
imaging science and signal processing \cite{BBDV09,BT09,MBNS17},
information theory and statistics \cite{CanTao07,CanTau05},
networks \cite{BG92},
traffic engineering \cite{Fuk84},
and many other fields where continuous optimization plays a major role.
In addition, \eqref{eq:Opt} also covers continuous relaxations of NP-hard discrete optimization problems ranging from the maximum clique problem to integer linear programming \cite{Bom97,BomSchUll17}.
As such, it should come as no surprise that \eqref{eq:Opt} has given rise to a thriving literature on iterative algorithmic methods aiming to reach an approximate solution in a reasonable amount of time.

Even though it is not possible to adequately review this literature here, we should point out that it includes methods as diverse as
quasi-Newton algorithms,
conditional gradient descent (Frank-Wolfe),
interior-point and active-set methods,
and
Bregman proximal/\acl{MD} schemes.
In particular, one very fruitful strategy for solving \eqref{eq:Opt} is to take a continuous-time viewpoint and design \acp{ODE} whose solution trajectories are ``negatively correlated'' with the gradient of $\obj$
\textendash\ see \eg \cite{AGR00,BT03,CEG09,ABS13,SBC14,AP16,Wib16,MerSta18} and references therein.
Doing so sheds new light on the properties of many algorithms proposed to solve \eqref{eq:Opt}, it provides Lyapunov functions to analyze their asymptotic behavior, and often leads to new classes of algorithms altogether.

A classical example of this heuristic arises in the study of dynamical systems derived from a \acdef{HR} metric, \ie a Riemannian metric induced by the Hessian of a Legendre-type function \cite{Dui01,BT03,ABB04,ABRT04}.
To make this more precise (see \cref{sec:Riemannian} for the details), the \acdef{HRGD} dynamics for \eqref{eq:Opt} can be stated as
\begin{equation}
\label{eq:HRGD}
\tag{HRGD}
\dot x
	= - \projmat(\point) \hmat(\point)^{-1} \nabla\obj(\point),
\end{equation}
where:
\begin{enumerate}
\addtolength{\itemsep}{\medskipamount}
\item
$\hmat(\point) = \Hess \hreg(\point)$
for some convex \emph{barrier function} $\hreg\from\clorthant\to\R\cup\{+\infty\}$ that satisfies a \emph{steepness} (or \emph{essential smoothness}) condition of the form
\begin{equation}
\label{eq:steep}
\lim_{k\to\infty} \norm{\nabla \hreg(\point^{\run})}_{2}
	= \infty
\end{equation}
for every sequence of interior points $\point^{\run}\in\orthant$ converging to the boundary $\bd(\clorthant)$ of $\clorthant$.
\item
$\projmat(\point)$ is the (Riemannian) projection map for the null space $\scrA_{0} = \ker\cmat \equiv \setdef{\point\in\Rn}{\cmat\point=0}$ of $\cmat$;
concretely, $\projmat(\point)$ has the closed-form expression
\begin{equation}
\label{eq:project}
\projmat(\point)
	= \Id - \hmat(\point)^{-1} \cmat^{\top}(\cmat\hmat(\point)^{-1}\cmat^{\top})^{-1}\cmat.
\end{equation}
\end{enumerate}

The intuition behind \eqref{eq:HRGD} is as simple as it is elegant:
to derive an interior-point method for \eqref{eq:Opt}, the positive orthant $\orthant$ is endowed with a Riemannian geometric structure that ``blows up'' near its boundary (\ie distances between points increase near the boundary).
In so doing, the (unrestricted) Riemannian gradient $\grad \obj(\point) = \hmat(\point)^{-1} \nabla\obj(\point)$ of $\obj$ becomes vanishingly small near the boundary of $\orthant$.
The projection map $\projmat(\point)$ further guarantees that the dynamics evolve in the affine hull $\scrA \equiv \setdef{\point\in\Rn}{\cmat\point=\cvec}$ of $\points$ (assumed throughout to be nonempty);
as a result, the solution trajectories of \eqref{eq:HRGD} \typo{starting in the relative interior $\intpoints$ of $\points$ remain in $\intpoints$ for all $t\geq0$ \cite{ABB04}.}

If the objective function $\obj$ is convex, \eqref{eq:HRGD} enjoys very robust convergence guarantees \cite{ABB04,ABRT04}, and many recent developments in acceleration techniques can also be traced back to this basic scheme \textendash\ \eg see \cite{Wib16} and references therein.
However, harvesting the full algorithmic potential of \eqref{eq:HRGD} also requires a suitable discretization of the dynamics in order to obtain a bona fide, implementable algorithm.
In \cite{AT04}, this was done \revise{by a discretization scheme} that ultimately gives rise to the \acdef{MD} update rule
\begin{equation}
\label{eq:MD}
\tag{MD}
\new\point
	= \argmin_{\pointalt\in\points} \{\step \nabla\obj(\point)^{\top} (\pointalt - \point) + \breg(\pointalt,\point)\},
\end{equation}
where
$\new\point\in\points$ denotes the algorithm's new state starting from $\point\in\points$,
$\step$ is the method's step-size,
and
$\breg$ denotes the \emph{Bregman divergence} of $\hreg$, \ie
\begin{flalign}
\label{eq:Bregman}
\breg(\pointalt,\point)
	= \hreg(\pointalt) - \hreg(\point) - \nabla \hreg(\point)^{\top}(\pointalt - \point).
\end{flalign}

First introduced by Nemirovski and Yudin \cite{NY83} for non-smooth problems, the \acl{MD} algorithm and its variants have met with prolific success in convex programming \cite{BecTeb03}, online and stochastic optimization \cite{SS11}, \aclp{VI} \cite{Nes09}, non-cooperative games \cite{BLM18,MZ19}, and many other fields of optimization theory and its applications.
Nevertheless, despite the appealing convergence properties of \eqref{eq:MD}, it is often difficult to calculate the \revise{update step} from $\point$ to $\new\point$ when the problem's feasible region $\points$ is not ``prox-friendly'' \textendash\
\ie when there is no efficient oracle for solving the convex optimization problem in \eqref{eq:MD} \cite{CoxJudNem14}.
With this in mind, our main goal in this paper is to provide a convergent, \revise{\emph{forward}} discretization of \eqref{eq:HRGD} which does not require solving a convex optimization problem at each update step.

\subsection*{Our contributions and prior work}

Our starting point is to consider an Euler discretization of \eqref{eq:HRGD} which we call the \acdef{HBA}, and which can be described by the update rule
\begin{equation}
\label{eq:HBA-intro}
\tag{HBA}
\new\point
	= \point - \step \projmat(\point) \hmat(\point)^{-1} \nabla \obj(\point).
\end{equation}
In the above, $\hmat(\point)$ and $\projmat(\point)$ are defined as in \eqref{eq:HRGD}, while the algorithm's step-size $\step \equiv \step(\point)$ is determined via an Armijo backtracking rule that we describe in detail in \cref{sec:algorithm}.
Before discussing our general results, we provide below a small sample of classical first-order schemes which can be seen as direct antecedents of \ac{HBA}:
\smallskip

\begin{example}[\acl{LV} systems]
\label{ex:LV}
Let $\nConsts=0$, so the feasible region of \eqref{eq:Opt} is the non-negative orthant $\clorthant=\Rnp$ of $\Rn$.
\typo{
If we set
\begin{equation}
\label{eq:Tsallis}
\theta(t)
	= \begin{cases}
	t\log t
		&\quad
		\text{for $\pexp=1$}
		\\
	\frac{1}{(2-\pexp)(1-\pexp)} t^{2-\pexp}
		&\quad
		\text{for $\pexp\in(1,2)$}
		\\
	-\log t
		&\quad
		\text{for $\pexp=2$}
	\end{cases}
\end{equation}
and $\hreg(\point) = \sum_{i=1}^{\nDims} \theta(\point_{i})$,
}
some straightforward algebra gives the \acl{LV} rule
\begin{equation}
\label{eq:LV}
\tag{LV}
\new\point_{i}
	= \point_{i} - \step \point_{i}^{\pexp} \pd_{i}\obj(\point),
\end{equation}
where we write $\pd_{i}\obj(\point)$ for the $i$-th partial derivative of $\obj$ at $\point$ (for simplicity, we are also dropping the dependence of $\step(\point)$ on $\point$).
For the convergence analysis of a special case of this system (modulo a regularization term), see \cite{AT04} and references therein.
\end{example}
\smallskip

\begin{example}[The \acl{RD}]
\label{ex:RD}
Let $A = (1,\dotsc,1) \in \R^{1\times\nDims}$ and $b=1$, so the feasible region of \eqref{eq:Opt} is the unit simplex $\points = \setdef{\point\in\Rnp}{\sum_{i} \point_{i} = 1}$.
If we take the negative entropy function $\hreg(\point) = \sum_{i=1}^{\nDims} \point_{i} \log \point_{i}$ stemming from the choice $\pexp=1$ above, a direct calculation yields \revise{$\hmat(\point) = \diag(1/\point_{1},\dotsc,1/\point_{\nDims})$} and \revise{$\projmat(\point) = \Id - \point\cdot(1,\dotsc,1)$}.
The induced \acl{HR} system is known as the \acdef{RD} and the corresponding incarnation of \eqref{eq:HBA-intro} takes the form
\begin{equation}
\label{eq:RD}
\tag{RD}
\txs
\new\point_{i}
	= \point_{i} - \step \point_{i} \bracks*{\pd_{\revise{i}}\obj(\point) - \sum\nolimits_{j=1}^{\nDims} \point_{j} \pd_{j}\obj(\point)}.
\end{equation}
The continuous-time version of \eqref{eq:RD} has a long history in evolutionary game theory \cite{HS03} and it has been successfully applied to a wide range of relaxations of NP-hard optimization problems \cite{Bom97a,BomSchUll17}.
\end{example}
\smallskip

\begin{example}[Affine scaling]
\label{ex:affine}
Suppose that $\obj(\point) = \cost^{\top}x$ for some cost vector $\cost\in\Rn$.
Then, defining $\hreg(\point)$ as in \cref{ex:LV}, we obtain the \acdef{AS} scheme
\begin{equation}
\label{eq:AS}
\tag{AS}
\new\point
	= \point - \step \bracks{\Id - X^{\pexp} \cmat^{\top} (\cmat X^{\pexp}\cmat^{\top})^{-1}\cmat} X^{\pexp} \cost
\end{equation}
where $X = \diag(\point_{1},\dotsc,\point_{\nDims})$.
The origins of \eqref{eq:AS} can be traced back to the work of Dikin in the 1960's and Karmarkar in the 1980's;
the convergence of the specific incarnation \eqref{eq:AS} was established in the seminal paper of Vanderbei et al. \cite{VanMekFre86}.
\end{example}
\smallskip

\begin{example}
[Regularized Newton methods]
Suppose that $\nConsts=0$ (so there are no equality constraints), and $\obj$ is convex and twice continuously differentiable.
Setting $\hreg(\point) = \obj(\point) + \frac{1}{2}\beta\norm{\point}_{2}^{2}$, we get $\hmat(\point) = \beta\Id + \Hess\obj(\point)$, leading in turn to the \acdef{RN} update rule
\begin{equation}
\label{eq:RN}
\tag{RN}
\new\point
	= \point - \step \bracks{\beta\Id + \Hess\obj(\point)}^{-1} \nabla\obj(\point)
\end{equation}
If $\obj$ is \emph{self-concordant} \cite{NN94}, the barrier function $\hreg(\point)$ satisfies the steepness requirement \eqref{eq:steep}, so \eqref{eq:RN} can be seen as a special case of \eqref{eq:HBA-intro}.
The convergence of this method was studied in detail in a recent paper by R.~A.~Polyak \cite{Pol09}.
\end{example}
\smallskip

The examples above show that \eqref{eq:HBA-intro} is a flexible method that covers several existing algorithms as special cases, and which can be easily tuned to the specifics of the problem at hand.
To analyze its asymptotic behavior, we introduce an Armijo backtracking procedure which guarantees ``sufficient decrease'' of the value of $\obj$ at each stage.
In so doing, we are able to show that the sequence $\point^{\run}$, $\run=\running$, of the algorithm's generated iterates converges to the set of \ac{KKT} points of \eqref{eq:Opt} under mild regularity assumptions on $\obj$ and a full row-rank assumption of the constraint matrix $\cmat$ (\cf \cref{thm:main}).
As an immediate corollary of this, we show that every limit point of \eqref{eq:HBA-intro} is a global minimum of $\obj$ if the objective function of \eqref{eq:Opt} is convex.
This global convergence result closes a significant open issue in the asymptotic analysis of Tseng et al.
\cite{TseBomSch11} for Armijo methods, where convergence of a replicator-type system is proved modulo a ``non-vanishing'' step-size hypothesis which cannot be verified directly from the problem's primitives.
As we show here, this step-size assumption is by no means harmless, and requires a delicate argument to establish.

In the special case where $\obj$ is quadratic (but otherwise possibly non-convex), we further show that $\obj(\point^{\run})$ converges at a sublinear rate of $\bigoh(1/k^{\rho})$ for some $\rho\in(0,1]$ depending only on the choice of the method's barrier function.
This shows that the chosen barrier function is a key design parameter for the convergence properties of \eqref{eq:HBA-intro};
we discuss this issue in detail in \cref{sec:rates}.

Finally, in \cref{sec:numerics}, we supplement our theoretical analysis by means of extensive numerical experiments with standard global optimization test functions (such as the Rosenbrock and Beale benchmarks), and we examine the method's observed convergence rate in a large-scale \aclp{TAP}.

\subsection*{Notation}

For all $\point\in\R^{\nDims}$, we will write $\diag(\point) \equiv \diag(\point_{1},\dotsc,\point_{\nDims})$ for the diagonal $\nDims\times\nDims$ matrix with the coordinates of $\point$ on the main diagonal. We set $S = \{1,2,\dotsc,\nDims\}$, and write $S_{\point}=\setdef{i\in S}{\point_{i} \neq 0}$ for the support of the vector $\point\in\Rn$.
For $\point\in\Rn$ and $J\subset S$, we let $\point_{J}=(\point_{j})_{j\in J}$ denote the restriction of $\point$ to the coordinates in the index set $J$.
Finally, we will write $\scrS^{\nDims}$, $\scrS_{+}^{\nDims}$ and $\scrS_{++}^{\nDims}$ for the space of real $\nDims\times\nDims$ symmetric, positive-semidefinite and positive-definite matrices respectively.

\section{Problem setup and preliminaries}
\label{sec:prelims}

\subsection{Definitions and assumptions}
\label{sec:blanket}

Throughout what follows, we will make the following blanket assumptions for \eqref{eq:Opt}:

\begin{assumption}
The objective $\obj\from\clorthant\to\R\cup\{+\infty\}$ of \eqref{eq:Opt} satisfies the following:
\label{asm:basic}
\begin{enumerate}
[\indent(\itshape a\hspace*{.5pt}\upshape)]
\addtolength{\itemsep}{\smallskipamount}
\item
\label{itm:smooth}
$\obj$ is proper and \ac{lsc} on $\clorthant$, continuously differentiable on $\points$, and $\nabla\obj$ is $L$-Lipschitz continuous on $\points$.
\item
\label{itm:bounded}
There exists some $\point^{\start}\in\intpoints$ such that the sublevel set $[\obj\leq \obj(\point^{\start})] \equiv \setdef{\point\in\points}{\obj(\point) \leq \obj(\point^{\start})}$ is bounded.
\end{enumerate}
\end{assumption}

\cref{asm:basic}\itemref{itm:bounded} is trivial when $\points$ is itself bounded;
moreover, taken together, Assumptions \ref{asm:basic}\itemref{itm:smooth} and \ref{asm:basic}\itemref{itm:bounded} imply that
\typo{the sublevel set $[\obj\leq \obj(\point^{\start})]$ is compact, so $\obj$ attains its minimum therein.}
These assumptions are quite standard in interior-point methods and, in particular, \acl{AS} schemes;
for an in-depth discussion, see \cite{GlHert93} and references therein.

To formulate the first-order optimality conditions for \eqref{eq:Opt}, consider the Lagrangian
\begin{flalign}
\label{eq:Lagrangian}
\typo{\mathcal{L}(\point,y,u)}
	= \obj(\point) - y^{\top}(\cmat\point - \cvec) - u^{\top}\point
\end{flalign}
where $y\in\R^{\nConsts}$ and $u\in\R^{\nDims}$ are the Lagrange multipliers corresponding respectively to the problem's equality and inequality constraints.
The \acf{KKT} conditions for \eqref{eq:Opt} may then be written as
\begin{equation}
\label{eq:KKT}
\tag{KKT}
\begin{aligned}
\nabla\obj(\point)
	&= \cmat^{\top} y + u
	\\
\cmat\point
	&= \cvec
	\\
u_{i} \point_{i} &= 0,\;
	u_{i} \geq0\;
	\text{for all $i=1,\dotsc,\nDims$}
\end{aligned}
\end{equation}
The set of all points $\sol\in\points$ for which the system \eqref{eq:KKT} admits a solution \typo{$(y,u)$} will be denoted in what follows by $\sols$.
As all constraints are linear, we do not need any constraint qualifications, and all local minima of $\obj$ also lie in $\sols$ by default.

Since the existence of a minimizer is guaranteed by \cref{asm:basic}, it follows that $\sols$ is nonempty.
Note also that, if $\sol\in\sols$, \typo{then} there exists some $\dsol \in\R^{m}$ such that
\begin{subequations}
\begin{flalign}
\nabla \obj(\sol) - \cmat^{\top}\dsol
	&\geq 0,
\label{eq:KKT1}
	\\
\diag(\sol) \parens{\nabla \obj(\sol) - \cmat^{\top}\dsol}
	&=0,
\label{eq:KKT2}
\end{flalign}
\end{subequations}
\typo{and vice versa.}

\subsection{Elements of Riemannian geometry}
\label{sec:Riemannian}

A key notion in our considerations is that of a \emph{Riemannian metric}, \ie a position-dependent variant of the ordinary (Euclidean) scalar product between vectors.
To define it, recall first that a \emph{scalar product} on $\R^{\nDims}$ is a symmetric, positive-definite bilinear form $\product{\cdot}{\cdot}\from \R^{\nDims}\times \R^{\nDims}\to \R$.%
\footnote{For a masterful introduction to Riemannian geometry, we refer the reader to \cite{Lee03}.}
This product defines a norm in the usual way and it can be represented equivalently via its \emph{metric tensor}, that is, a symmetric, positive-definite matrix $\gmat\in\scrS_{++}^{\nDims}$ with components
\begin{equation}
\label{eq:metric-components}
\gmat_{ij}
	= \product{\bvec_{i}}{\bvec_{j}}
\end{equation}
\typo{in the standard basis} $\{\bvec_i\}_{i=1}^{\nDims}$ of $\R^{\nDims}$.
A \emph{Riemannian metric} on a nonempty open set $\open\subseteq\R^{\nDims}$ is then defined to be a smooth assignment of scalar products $\product{\cdot}{\cdot}_{\point}$ to each $\point\in\open$ \textendash\ or, equivalently, a smooth field $\gmat(\point)$ of symmetric positive-definite matrices on $\open$.

Given a Riemannian metric on $\open$, the \emph{Riemannian gradient} of a smooth function $\test\from\open\to\R$ at $\point\in\open$ is defined via the characterization
\begin{equation}
\label{eq:grad-char}
\braket{\grad\test(\point)}{\tvec}_{\point}
	= \test'(\point;\tvec)
	\quad
	\text{for all $\tvec\in\R^{\nDims}$},
\end{equation}
where $\test'(\point;\tvec) = \left.\frac{d}{dt}\right\vert_{t=0^{+}} \test(\point+t\tvec)$ denotes the directional derivative of $\test$ at $\point$ along $\tvec$.
More concretely, by expressing everything in components, it is easy to see that
$\grad\test(\point)$ is given by the explicit expression
\begin{equation}
\label{eq:grad-def}
\grad\test(\point)
	= \hmat(\point)^{-1} \nabla\test(\point).
\end{equation}

Bringing the above closer to our setting, let $\vecspace_{0}\subseteq\R^{\nDims}$ be a subspace of $\R^{n}$ and let $\vecspace$ be an affine translate of $\vecspace_{0}$ such that $\open_{0} \equiv \open\cap\vecspace$ is nonempty.
Then, viewing $\open_{0}$ as an open subset of $\vecspace$, the \emph{gradient of $\test$ restricted to $\open_{0}$} is defined as the unique vector $\grad_{\open_{0}} \test(\point) \equiv \grad \test\vert_{\open_{0}}(\point) \in \vecspace_{0}$ such that
\begin{equation}
\label{eq:grad-res}
\typo{\product{\grad_{\open_{0}}\test(\point)}{\tvec}_{\point}}
	= \test'(\point;\tvec)
	\quad
	\text{for all $\tvec\in\vecspace_{0}$}.
\end{equation}
Hence, specializing all this to the problem at hand, let $\gmat(\point)$ be a Riemannian metric on the open orthant $\orthant=\R_{++}^{\nDims}$ of $\R^{\nDims}$ and set
\begin{equation}
\label{eq:hulls}
\begin{aligned}
\scrA_{0}
	&= \ker \cmat
	= \setdef{\point\in\R^{\nDims}}{\cmat\point=0},
	\\
\scrA
	&= \setdef{\point\in\R^{\nDims}}{\cmat\point = \cvec},
\end{aligned}
\end{equation}
as in \cref{sec:introduction}.
Then, a straightforward exercise in matrix algebra shows that the gradient of $\obj$ restricted to $\intpoints = \orthant\cap\scrA$ can be written in closed form as
\begin{equation}
\label{eq:grad-res-coords}
\grad_{\intpoints} \obj(\point)
	= \projmat(\point) \gmat(\point)^{-1} \nabla\obj(\point)
\end{equation}
with $\projmat(\point)$ defined as in \eqref{eq:project}, \ie $\projmat(\point) = \Id - \gmat(\point)^{-1}\cmat^{\top}(\cmat\gmat(\point)^{-1}\cmat^{\top})^{-1}\cmat$.

To streamline notation for later, we will denote the negative (restricted) gradient of $\obj$ at $\point\in\intpoints$ as
\begin{equation}
\label{eq:v}
\vecfield(\point)
	= -\grad_{\intpoints} \obj(\point)
	= - \projmat(\point) \hmat(\point)^{-1} \nabla\obj(\point).
\end{equation}
Defined this way, $\vecfield(\point)$ corresponds to the direction of steepest descent of $\obj$ along $\points$ relative to the metric $\hmat(\point)$.
In particular, since $\vecfield(\point)\in\scrA_{0}$ for all $\point\in C$, it follows that
\begin{equation}
\label{eq:angle}
-\nabla \obj(\point)^{\top}\vecfield(\point)
	= \norm{\vecfield(\point)}_{\point}^{2},
\end{equation}
where, in obvious notation, we let $\norm{\tvec}_{\point}^{2} = \product{\tvec}{\tvec}_{\point}$ for all $\tvec\in\scrA_{0}$.

\subsection{\acl{HR} metrics}
\label{sec:HR}

A very important class of Riemannian metrics (and the main focus of our paper) can be generated by taking the Hessian of a smooth convex function.
More precisely:

\begin{definition}
\label{def:HR}
We say that \typo{$\hreg\from\clorthant\to\R\cup\{+\infty\}$} is a \emph{barrier} \textpar{or \emph{metric generating}} \emph{function} if
\begin{enumerate}
\item
$\hreg$ is twice continuously differentiable on $\orthant$.
\item
The Hessian $\Hess \hreg$ of $\hreg$ is locally Lipschitz continuous and positive-definite on $\orthant$.
\item
$\norm{\pd_{i}\hreg(\point^{\run})}_{2} \to \infty$ for every sequence of interior points $\point^{\run}\in\orthant$ converging to the boundary $\bd(\clorthant)$ of $\clorthant$.
\end{enumerate}
\smallskip
If $\hreg$ is a barrier function as above, the \acdef{HR} metric induced by $\hreg$ is defined as
\begin{equation}
\label{eq:HR}
\hmat(\point)
	= \Hess \hreg(\point)
	\quad
	\text{for all $\point\in\orthant$}.
\end{equation}
\end{definition}
\smallskip

\begin{remark}
The systematic study of \acl{HR} metrics dates back at least to Duistermaat \cite{Dui01}.
In the context of convex programming, these metrics were popularized by the authors of \cite{BT03,ABB04,ABRT04} who introduced the \acl{HR} gradient dynamics \eqref{eq:HRGD} discussed in \cref{sec:introduction}.
With regard to terminology, \cref{def:HR} essentially follows the setup of \cite{ABB04} with a number of simplifications aimed to take advantage of the specific structure of the non-negative orthant.
\end{remark}
\smallskip

\begin{remark}
Up to mild differences, the notion of a barrier function essentially coincides with that of a \ac{DGF} as used to derive the \acl{MD} algorithm \cite{NY83,Nes09}.
A detailed discussion of the connections between \acl{HR} metrics and \acl{MD} would take us too far afield, so we refer the reader to \cite{ABB04,AusTeb06} for a more general treatment.
\end{remark}
\smallskip

A systematic way of constructing barrier functions on $\orthant$ is to take separable sums of the form
\begin{equation}
\label{eq:separable}
\hreg(\point)
	= \sum_{i=1}^{\nDims} \theta_{i}(\point_{i})
\end{equation}
where each function $\theta_{i}\from(0,\infty)\to\R$ is a barrier function on $(0,\infty) = \R_{++}$ (viewed here as the positive orthant of $\R$).
For technical reasons, it will be convenient to assume two further conditions for $\theta_{i}$, leading to the following definition:

\begin{definition}
\label{def:kernel}
We say that $\theta\from(0,\infty)\to\R$ is a \emph{metric-inducing kernel} if:
\smallskip
\begin{enumerate}
[\indent(\itshape a\hspace*{.5pt}\upshape)]
\addtolength{\itemsep}{\smallskipamount}
\item
\label{itm:barrier}
$\theta$ is twice continuously differentiable on $(0,\infty)$, $\theta''$ is positve and locally Lipschitz continuous on $(0,\infty)$, and $\lim_{t\to0^{+}} \theta'(t) = -\infty$.
\item
\label{itm:strong}
$\inf_{t>0} \theta''(t) > 0$, \ie $\theta''(t)\geq\beta$ for some $\beta>0$ and all $t\in(0,\infty)$.
\item
\label{itm:growth}
$\inf_{t>0} t\theta''(t)>0$, \ie $t\theta''(t)\geq\eps$ for some $\eps>0$ and all $t\in(0,\infty)$.
\end{enumerate}
\end{definition}
\smallskip

Of the above requirements, \itemref{itm:barrier} simply specializes the barrier function requirements of \cref{def:HR} to $(0,\infty)$.
Requirement \itemref{itm:strong} strengthens the strict convexity assumption by essentially positing strong convexity over $(0,\infty)$;
this assumption can be dropped altogether, but we use it to simplify our arguments later on.%
\footnote{If $\points$ is compact, it suffices to have $\inf_{t} \theta''(t) > 0$ on any compact subset of $(0,\infty)$, and this holds trivially by the positivity and continuity of $\theta''$.
In the general case, the boundedness requirement of \cref{asm:basic}\itemref{itm:bounded} can be used to a similar effect because all our analysis takes place in the sublevel set $\setdef{\point\in\points}{\obj(\point) \leq \obj(\point^{\start})}$.}
Finally, \itemref{itm:growth} essentially posits that $\theta''(t)$ grows at least as $\bigoh(1/t)$ as $t\searrow0^{+}$.
This ``sufficient growth'' requirement plays an important technical role later on in our analysis but is relatively mild otherwise.%
\footnote{Coupled with the requirement $\lim_{t\to0^{+}} \theta'(t) = -\infty$, the growth condition \itemref{itm:growth} only fails for fringe examples such as $\theta''(t) = 1/(t\log t)$ and the like.}

For concreteness, we provide some standard examples of kernel functions below:
\smallskip
\begin{enumerate}
\addtolength{\itemsep}{\smallskipamount}

\item{\makebox[15em]{\emph{Regularized Gibbs entropy:}\hfill}}
	$\theta(t) = \frac{1}{2}\beta t^{2} + t \log t$.

\item{\makebox[15em]{\emph{Regularized Tsallis entropy:}\hfill}}
	$\theta(t) = \frac{1}{2}\beta t^{2} + \frac{1}{(1-\pexp)(2-\pexp)} t^{2-\pexp}$, $\pexp\in(1,2)$.

\item{\makebox[15em]{\itshape Regularized log-barrier \textpar{Burg}:\hfill}}
	$\theta(t)= \frac{1}{2}\beta t^{2} - \log t$.
\end{enumerate}
\smallskip

The above examples only provide a snapshot of possible choices;
for more examples, see \cite{ABB04,MerSan18}.
We should also note that the regularization term $\frac{1}{2}\beta t^{2}$ is only included to guarantee that $\inf_{t}\theta''(t) \geq \beta$.
As we discussed above, this requirement can be dropped, corresponding to the baseline case $\beta=0$ (the examples we presented in the introduction were all taken with $\beta=0$).
It is also clear that these functions can be combined to generate mixture functions preserving the defining properties of a metric-inducing kernel.
For instance, modulo the regularization term $\frac{1}{2}\beta t^{2}$, Tseng et al. \cite{TseBomSch11} considered the mixture
\begin{equation}
\typo{\theta_{\gamma}(t)}
	= \frac{1}{2} \beta t^{2}
	+ \begin{cases}
	t\log t - t
		&\quad
		\text{if $\gamma=1/2$},
		\\
	\frac{1}{2(1-\gamma)(1-2\gamma)} t^{2(1-\gamma)}
		&\quad
		\text{if $\gamma\in(1/2,1)$},
		\\
	-\log t
		&\quad
		\text{if $\gamma=1$},
	\end{cases}
\end{equation}
which provides a continuous \typo{homotopy} interpolation \typo{of $1/\theta_\gamma''(t)$} between the Gibbs and Burg kernels for $\gamma=1/2$ and $\gamma=1$ respectively (the range $0<\gamma<1/2$ is not considered here because it violates the steepness requirement $\lim_{t\searrow0^{+}}\theta'(t) = -\infty$).

The benefit of using a metric-inducing kernel as above is that the resulting \acl{HR} metric takes the convenient diagonal form
\begin{equation}
\hmat(\point)
	= \diag(\theta_{1}''(\point_{1}),\dotsc,\theta_{\nDims}''(\point_{\nDims}))
\end{equation}
which leads to the straightforward expression $\hmat(\point)^{-1} = \diag(1/\theta_{1}''(\point_{1}),\dotsc,1/\theta''(\point_{\nDims}))$.
By \cref{def:kernel}\itemref{itm:growth}, the inverse matrix $\hmat(\point)^{-1}$ can be extended continuously to the boundary $\bd(\clorthant)$ of $\clorthant$ in the obvious way, and its explicit diagonal form greatly facilitates our analysis in the next sections.
Unless explicitly mentioned otherwise, all \acl{HR} metrics in what follows will be assumed to come from a kernel function as above;
for a more general treatment, see \cite{ABB04}.

\section{The \acl{HBA}}
\label{sec:algorithm}

Viewed abstractly, the \acl{HBA} can be formulated as a recursive update rule of the general form
\begin{equation}
\label{eq:update}
\new\point
	= \point + \step\tvec.
\end{equation}
Specifically, given an input state $\point\in\points$, a new state $\new\point\in\points$ is produced by taking a step along the \revise{tangent search direction $\tvec\in\scrA_{0}$}, properly scaled by the step-size $\step>0$.
In the rest of this section, we discuss in detail the definition of the search direction $\tvec$ and the step-size $\step$.

\subsection{The search direction}
\label{sec:search}

Given a Hessian Riemannian metric $\hmat(\point) \equiv \Hess \hreg(\point)$ on $\intpoints$, the algorithm's search direction will be determined by solving a quadratic optimization problem of the form
\begin{equation}
\begin{aligned}
\label{eq:cost-gain}
\textrm{minimize}
	&\quad
	\nabla\obj(\point)^{\top} \tvec
	+ \frac{1}{2} \norm{\tvec}_{\point}^{2}
	\\
\textrm{subject to}
	&\quad
	\cmat\tvec = 0,
\end{aligned}
\end{equation}
with the norm $\norm{\cdot}_{\point}$ prescribed by some \acl{HR} metric on $\orthant$ as in the previous section.
Heuristically, the linear term $\nabla\obj(\point)^{\top} \tvec \equiv \obj'(\point;\tvec)$ simply captures the corresponding first-order change in the value of $\obj$ along $\tvec$;
analogously, the quadratic term in \eqref{eq:cost-gain} can be interpreted as a ``cost of motion'' along $\tvec$.
As such, \eqref{eq:cost-gain} identifies the direction of steepest descent modulo the cost of taking said step.%
\footnote{For a game-theoretic analogue of this idea, see \cite{MerSan18}.}

From an algebraic standpoint, a standard calculation shows that the solution of \eqref{eq:cost-gain} is simply the (negative) \acl{HR} gradient of $\obj$ at $\point$, \ie it is equal to
\begin{equation}
\vecfield(\point)
	\equiv -\grad_{\points} \obj(\point)
	= -\projmat(\point) \hmat(\point)^{-1} \nabla \obj(\point).
\end{equation}
Perhaps more intuitively, this search direction also coincides with the solution of the trust-region problem
\begin{equation}
\label{eq:trust}
\begin{aligned}
\textrm{minimize}
	&\quad
	\nabla\obj(\point)^{\top} \tvec
	\\
\textrm{subject to}
	&\quad
	\cmat\tvec = 0,\;
	\norm{\tvec}_{\point} \leq r
\end{aligned}
\end{equation}
when $r>0$ is large enough.%
\footnote{In particular, it suffices to take $r$ equal to the minimum value of \eqref{eq:cost-gain}.}
The above shows that a search vector chosen in this way maximizes the first-order decrease in the value of $\obj$ over all vectors with bounded norm.
In turn, this exhibits the close connection of Hessian Riemannian descent methods to interior-point trust-region methods as in \cite{BonPol199,HaeLiuYe18};
we will return to this point later.

We close this section with the straightforward observation that the zeros of the search direction $\vecfield(\point)$ correspond precisely to the critical points of \eqref{eq:Opt}:

\begin{lemma}
\label{lem:critical}
For all $\point\in\intpoints$, we have $\vecfield(\point)=0$ if and only if $\nabla\obj(\point)\in\scrA_{0}^{\bot} \equiv \image(A^{\top})$.
\end{lemma}

The proof of \cref{lem:critical} is an elementary consequence of the definition of $\vecfield(\point)$, so we omit it.
We only mention this result here to highlight the fact that the update rule \eqref{eq:update} with search direction $\vecfield(\point)$ remains stationary if the input state $\point$ is a zero of $\vecfield(\point)$.
In what follows, we use this fact freely without referring to it explicitly.

\subsection{The method's step-size}
\label{sec:step}

The main challenge in setting the method's step-size is twofold:
\begin{inparaenum}
[\itshape a\upshape)]
\item
we need to guarantee that $\new\point$ is feasible for all input states $\point\in\intpoints$;
and
\item
the method should exhibit ``sufficient decrease'' in the sense that $\obj(\new\point)$ is sufficiently smaller than $\obj(\point)$ at each step.
\end{inparaenum}

We begin with the issue of feasibility.
To that end, adopting terminology which is common in the affine scaling literature, consider the ``dual variable''
\begin{flalign}
\label{eq:dual}
y(\point)
	= (\cmat\hmat(\point)^{-1}\cmat^{\top})^{-1}\cmat\hmat(\point)^{-1}\nabla \obj(\point)
\end{flalign}
and the ``reduced cost''
\begin{flalign}
\label{eq:reduced}
r(\point)
	= \nabla \obj(\point) - \cmat^{\top}y(\point)
	= -\hmat(\point)\vecfield(\point).
\end{flalign}
Since the Hessian $\hmat(\point)$ is diagonal \typo{by construction}, we can use the reduced cost vector $r(\point)$ to rewrite the update rule \eqref{eq:update} in components as
\begin{flalign}
\new\point_{i}
	= \point_{i} - \step(\point) \frac{r_{i}(\point)}{\theta_{i}''(\point_{i})}
	= \point_{i} \parens*{1 - \frac{\step(\point) r_{i}(\point)}{\point_{i}\theta_{i}''(\point_{i})}}.
\end{flalign}
Consequently, we will have $\new\point_{i} > 0$ if \typo{either $r_{i}(\point)\leq0$ or else}
\begin{equation}
\step(\point)
	< \frac{x_{i}\theta_{i}''(\point_{i})}{r_{i}(\point)}.
\end{equation}
Hence, to guarantee feasibility, it suffices to take $\step(\point) < \step_{0}(\point)$ where
\begin{equation}
\label{eq:alpha0}
\step_{0}(\point)
	= \min_{i=1,\dotsc,\nDims}\setdef{\point_{i}\theta_{i}''(\point_{i})/r_{i}(\point)}{r_{i}(\point) > 0},
\end{equation}
with the usual convention $\min\varnothing = \infty$.

Now, to decrease the value of the objective function at each step of the algorithm, our starting point will be the well-known descent inequality \cite{Pol87}
\begin{equation}
\label{eq:descent}
\obj(\pointalt) - \obj(\point)
	\leq \nabla \obj(\point)^{\top}(\pointalt - \point)
	+ \frac{L}{2} \norm{\pointalt - \point}_{2}^{2},
\end{equation}
which holds for all $\point,\pointalt\in\points$.
Then, taking $\pointalt = \point + \lambda \vecfield(\point)$ in \eqref{eq:descent} and using the angle relation \eqref{eq:angle}, we get
\begin{flalign}
\obj(\point + \lambda \vecfield(\point)) - \obj(\point)
	&\leq -\lambda \norm{\vecfield(\point)}_{\point}^{2}
	+ \frac{1}{2} \lambda ^{2} L \norm{\vecfield(\point)}_{2}^{2}
	\notag\\
	&\leq -\beta\lambda \norm{\vecfield(\point)}_{2}^{2}
	+ \frac{1}{2} \lambda ^{2} L \norm{\vecfield(\point)}_{2}^{2}
	\notag\\
	&= - \beta\lambda \, \parens*{1 - \frac{\lambda L}{2\beta}} \, \norm{\vecfield(\point)}_{2}^{2},
\end{flalign}
where, in the second line, we used the fact that $\norm{\tvec}_{\point}^{2} = \tvec^{\top} \hmat(\point) \tvec \geq \beta \tvec^{\top} \tvec = \beta \norm{\tvec}_{2}^{2}$.

In view of the above, feasibility and descent are both guaranteed as long as the step-size $\step(\point)$ of the method \typo{at the point} $\point\in\points$ is less than $\min\{\step_{0}(\point),2\beta/L\}$.
To proceed, we will further employ an Armijo backtracking procedure to guarantee \emph{sufficient decrease}, \ie that
\begin{equation}
\label{eq:sufficient}
\obj(\new\point)
	\leq \obj(\point) - \armijo\cdot\step(\point) \norm{\vecfield(\point)}_{\point}^{2}
\end{equation}
for some $\armijo\in(0,1)$.
To achieve this, we bootstrap the process with the step-size
\begin{equation}
\label{eq:step-start}
\underline\step(\point)
	= \min\{\step_{0}(\point),2\beta/L\}.
\end{equation}
If \eqref{eq:sufficient} is satisfied with $\step(\point) = \underline\step(\point)$, we will accept the iterate $\new\point$ generated from \eqref{eq:update};
otherwise, we shrink the step-size $\underline\step(\point)$  by a factor of $\shrink\in(0,1)$, and we keep backtracking until \eqref{eq:sufficient} is satisfied.%
\footnote{In practice, $\armijo$ is chosen very small (around $10^{-4}$), while typical values for $\shrink$ lie in the range between $0.1$ and $0.5$ \cite{NocWri00}.}
Formally, this means that the step-size of the method will be of the form $\step(\point) = \shrink^{\ell} \underline\step(\point)$ where  $\ell\geq0$ is the first nonnegative integer such that
\begin{equation}
\label{eq:Armijo}
\obj(\point + \shrink^{\ell}\underline\step(\point) \vecfield(\point)) - \obj(\point)
	\leq -\armijo \shrink^{\ell} \underline\step(\point) \norm{\vecfield(\point)}_{\point}^{2}.
\end{equation}
\cref{lem:step} in the next section shows that this backtracking process terminates after a finite number of steps.
In this way, we obtain a well-defined step-size policy which simultaneously guarantees feasibility and sufficient decrease.

\subsection{The \acl{HBA}}
\label{sec:HBA}

\typo{Combining} all of the above, the \acdef{HBA} can be stated in recursive form as
\begin{equation}
\label{eq:HBA}
\tag{HBA}
\point^{\run+1}
	= \point^{\run}
	- \step(\point^{\run}) \projmat(\point^{\run}) \hmat(\point^{\run})^{-1} \nabla\obj(\point^{\run})
\end{equation}
where
\begin{enumerate}
\item
$\run = \running$, is the algorithm's iteration counter.
\item
$\point^{\run}$ denotes the state of the algorithm at step $\run$;
the algorithm is initialized at a point $\point^{0}$ satisfying \cref{asm:basic}\itemref{itm:bounded}.
\item
$\step(\point)$ is the algorithm's step-size at state $\point$, defined implicitly via the Armijo backtracking process described in the previous section.
\item
$\projmat(\point)$ and $\hmat(\point)$ are determined by a \acl{HR} metric chosen by the optimizer (\cf \cref{sec:HR}).
\end{enumerate}
For a pseudocode implementation of \eqref{eq:HBA}, see \cref{alg:HBA}.
\smallskip


\begin{algorithm}[tbp]
\caption{\acf{HBA}}
\label{alg:HBA}

\tt
\begin{algorithmic}[1]
\Require
	sufficient decrease factor $\armijo\in(0,1)$,
	shrink factor $\shrink\in(0,1)$
	\vspace{1ex}
\State
	initialize $x\in\points$
	\Comment{initialization}
\While{stopping criterion not satisfied}
	\State
		$v \leftarrow -\grad_{\points} \obj(x)$
		\Comment{search direction}
	\State
		$\step \typo{\,\leftarrow} \min\{\step_{0}(x),2\beta/L\}$
		\Comment{set step-size}
	\State
		$x^{+} \leftarrow x + \step v$
		\Comment{set test point}
	\While{$\obj(x^{+}) > \obj(x) - \armijo\step\norm{v}_{x}^{2}$}
		\Comment{suff.\hspace{-1.5ex} decrease?}
		\State
			$\step \leftarrow \shrink\step$
			\Comment{shrink step-size}
		\State
			$x^{+} \leftarrow x + \step v$
			\Comment{update test point}
		\EndWhile
	\State
		$x \leftarrow x^{+}$
		\Comment{new state}
\EndWhile
\State
	\Return $x$
\end{algorithmic}
\end{algorithm}


\typo{Importantly,} even though \eqref{eq:HBA} looks similar to the interior gradient methods of \cite{AusTeb06,AusSilTeb07}, the actual update steps performed are fundamentally different.
Specifically, the gradient method of Auslender and Teboulle \cite{AusTeb06} performs at each iteration a prox-step using a Bregman function to ensure that the algorithm's iterates remain in the problem's feasible region \textendash\ recall the definition of \eqref{eq:MD}.
This approach implicitly assumes that the problem's constraint set is sufficiently ``simple'' for the Bregman proximal step to be performed in a computationally efficient way;
\eqref{eq:HBA} does not require a prox-step, so it is more lightweight in that respect.
\section{Global convergence analysis}
\label{sec:global}

To present our convergence analysis, two more definitions are required.
Specifically, if $\point^{\run}$, $\run = \running$, is the sequence of iterates generated by \eqref{eq:HBA}, we write
\begin{flalign}
\scrL
	&= \setdef{\hat\point\in\points}{\text{some subsequence $\point^{\run_{\runalt}}$ of $\point^{\run}$ converges to $\hat\point$}}
\intertext{for the set of limit points of the algorithm, and we let}
\Lambda
	&= \setdef{\hat\point\in\points}{\text{$\lim_{k\to\infty}\obj(\point^{\run}) = \obj(\hat\point)$ and $\diag(\hat\point) r(\hat\point) = 0$}}
\end{flalign}
Our main convergence result may then be stated as follows:

\begin{theorem}
\label{thm:main}
With notation as above, we have:
\begin{enumerate}
[\indent\upshape(\itshape a\hspace*{.5pt}\upshape)]
\item
The sequence $\point^{\run}$ is bounded and $\obj(\point^{\run})$ is non-increasing.
\item
Every point $\sol\in\scrL$ satisfies complementarity in the sense that $r_{i}(\sol) = 0$ whenever $\sol_{i} > 0$.
In particular, $\scrL\subseteq\Lambda$, so $\obj(\point^{\run})$ converges.
\item
Every limit point of $\point^{\run}$ is a \ac{KKT} point of $\obj$, provided one of the following conditions holds:
\begin{enumerate}
[\upshape(1)]
\item
\label{itm:cvx}
$\obj$ is convex;
in this case $\point^{\run}$ converges to $\argmin\obj$.
\item
\label{itm:isolated}
$\Lambda$ consists of isolated points.
\item
\label{itm:strict}
Every point in $\Lambda$ satisfies strict complementarity, \ie $\point_{i} + r_{i}(\point) > 0$ for all $i\in S=\{1,\dotsc,\nDims\}$.
\end{enumerate}
\end{enumerate}
\end{theorem}

\cref{thm:main} can be seen as the bona fide, algorithmic analogue of the continuous-time analysis of Alvarez et al. \cite{ABB04} of \acl{HR} gradient flows.
To the best of our knowledge, the closest result of this type in the literature is the convergence analysis of Tseng et al. \cite{TseBomSch11} for a replicator-type descent algorithm applied to quadratic programs in standard from.
However, the results of \cite{TseBomSch11} rely crucially on the assumption that the algorithm's step-size does not become vanishingly small in the limit:
this assumption is a major obstacle to the applicability of the analysis of \cite{TseBomSch11}, as there is no way to verify it from the problem's primitives.
Dropping this assumption requires a delicate \textendash\ and intricate \textendash\ argument which takes up the first part of the remainder of this section.

\subsection{Step-size analysis}
\label{sec:step-analysis}

As stated above, our main goal in what follows is to show that the algorithm's step-size sequence $\step^{\run} \equiv \step(\point^{\run})$ is bounded away from zero.
We begin with a trivial upper bound which we state only for completeness:

\begin{lemma}
\label{lem:upper}
The step-size sequence $\step^{\run} \equiv \step(\point^{\run})$ of \eqref{eq:HBA} satisfies $\sup_{\run} \step^{\run} < \infty$.
\end{lemma}

To get a lower bound for the algorithm's step-size, we begin by showing that the ``bootstrap'' step-size $\underline\step(\point)$ of \eqref{eq:step-start} is itself bounded away from zero.
\revise{In the context of affine scaling algorithms for linear programming, similar results have been proven in the special case where the Riemannian geometry is generated by the log-barrier kernel (the Burg entropy);
see \cite{LagVan90} for an early result in this direction.%
\footnote{We thank an anonymous referee for mentioning this reference to us.}
This kernel gives rise to very convenient closed-form expressions that greatly simplify the calculations;
however, for the general framework considered here, we need a fairly intricate analysis that cannot be handled by the derivations of \cite{LagVan90}.
We present the relevant calculations below:}

\begin{lemma}
\label{lem:lower}
\typo{We have $\inf\setdef{\underline\step(\point)}{\point\in\intpoints,\obj(\point)\leq \obj(\point^{0})}>0$.}
\end{lemma}

\begin{proof}
Since $\underline\step(\point) = \min\{\step_{0}(\point),2\beta/L\}$, it suffices to show that
$\inf\setdef{\step_{0}(\point)}{\point\in\intpoints,\obj(\point)\leq \obj(\point^{\start})}>0$.
In turn, by the definition \eqref{eq:alpha0} of the function $\step_{0}(\point)$, it suffices to look at points $\point$ for which $r_{i}(\point)>0$ for some $i=1,\dotsc,\nDims$.
We thus have to bound the quantity $\point_{i}\theta_{i}''(\point_{i})/r_{i}(\point)$ \revise{away from zero}, which, by \cref{def:kernel}, boils down to \revise{showing that $r_{i}(\point)$ is bounded from above}.

Since $r(\point)=\nabla \obj(\point)-\cmat^{\top}y(\point)$, this is achieved once we have an \revise{upper bound} for the ``dual variable'' $y(\point)=(\cmat\hmat(\point)^{-1}\cmat^{\top})^{-1}\cmat\hmat(\point)^{-1}\nabla \obj(\point)$ defined in \eqref{eq:dual}.
To achieve this, define the matrix $M_{\point} = \cmat\hmat(\point)^{-1/2} \in \R^{\nConsts\times\nDims}$, so $y(\point)$ is the unique solution to the linear system
\begin{equation}
M_{\point}M_{\point}^{\top}y
	= M_{\point}\hmat(\point)^{-1/2} \nabla \obj(\point).
\end{equation}
By Cramer's rule, we can explicitly compute the $i$-th coordinate of the vector $y(\point)$ via the formula
\begin{flalign}
y_{i}(\point)
	= \frac{\det\left((M_{\point}M_{\point}^{\top})^{1},\dotsc,M_{\point} \hmat(\point)^{-1/2}\nabla \obj(\point),\dotsc,(M_{\point}M_{\point}^{\top})^{\nConsts}\right)}{\det(M_{\point}M_{\point}^{\top})}.
\end{flalign}
This can be simplified by some straightforward, albeit tedious, algebraic manipulations. 
Indeed, for a matrix $\cmat\in\R^{m\times n}$ let
\begin{equation}
\cmat^{k_{1},\dotsc,k_{\nConsts}}
	= (a^{k_{1}},\dotsc,a^{k_{\nConsts}}),
	\qquad
	1\leq k_{1}< k_{2}< \dotsm <k_{\nConsts}\leq n,
\end{equation}
denote the $\nConsts \times \nConsts$ matrix obtained from the columns $a^{k_{1}},\dotsc,a^{k_{\nConsts}}$ of $\cmat$. 
By the Cauchy-Binet formula, we can compute the denominator as
\begin{flalign}
\det(M_{\point}M_{\point}^{\top})
	&=\sum_{1\leq k_{1}<\dotsm< k_{\nConsts}\leq n}\det\left(M_{\point}^{k_{1},\dotsc,k_{\nConsts}}\right)^{2}.
\end{flalign}
Since the Hessian matrix $\hmat(\point)$ is diagonal, it is immediate that
\begin{equation}
M_{\point}
	=\left[\theta_{1}''(\point_{1})^{-1/2}a^{1},\dotsc,\theta_{\nDims}''(\point_{\nDims})^{-1/2}a^{\nDims}\right],
\end{equation}
implying in turn that $M_{\point}$ can be extended continuously to the entire orthant $\Rn_{+}$ via the convention $1/\infty = 0$. 
We thus get
\begin{flalign}
\det\left(M_{\point}M_{\point}^{\top}\right)
	=\sum_{1\leq k_{1}<\dotsm< k_{\nConsts}\leq n}
		\frac{1}{\theta_{k_{1}}''(\point_{k_{1}})\cdots\theta''_{k_{\nConsts}}(\point_{k_{\nConsts}})}\det(\cmat^{k_{1},\dotsc,k_{\nConsts}})^{2}.
\end{flalign}
In a similar fashion we can express the numerator as the determinant of a matrix product between the matrices
$\cmat_{\point} = \left[a^{1}/\theta_{1}''(\point_{1}),\dotsc,a^{\nDims}/\theta_{\nDims}''(\point_{\nDims})\right]$, and $B_{\point,i}^{\top} = \left[a_{1},\dotsc,\nabla \obj(\point),\dotsc,a_{\nConsts}\right]$. 
Then, applying the Cauchy-Binet formula again, we obtain
\begin{flalign}
y_{i}(\point)
	= \frac{\sum_{1\leq k_{1}<\dotsm< k_{\nConsts}\leq n}
			\theta_{k_{1}}''(\point_{k_{1}})^{-1} \dotsm \theta''_{k_{\nConsts}}(\point_{k_{\nConsts}})^{-1}
			\det(\cmat^{k_{1},\dotsc,k_{\nConsts}})\det(B_{\point,i}^{k_{1},\dotsc,k_{\nConsts}})}
		{\sum_{1\leq k_{1}<\dotsm< k_{\nConsts}\leq n}
			\theta_{k_{1}}''(\point_{k_{1}})^{-1} \dotsm \theta''_{k_{\nConsts}}(\point_{k_{\nConsts}})^{-1}
			\det(\cmat^{k_{1},\dotsc,k_{\nConsts}})^{2}}.
\end{flalign}
Since $\cmat$ has full rank, the above is well defined.

\revise{To establish an upper bound for} this last expression, we use the simple inequality
\begin{flalign}
\frac{\abs{\sum_{i=1}^{\nDims}b_{i}}}{\sum_{i=1}^{\nDims}\sigma_{i}}\leq\max_{i}\frac{\abs{b_{i}}}{\sigma_{i}}
\end{flalign}
for $\sigma_{i}>0$. 
We then get
\begin{flalign}
\abs{y_{i}(\point)}
	\leq\max\abs*{\frac{\det(B_{\point,i}^{k_{1},\dotsc,k_{\nConsts}})}{\det(\cmat^{k_{1},\dotsc,k_{\nConsts}})}}
	=: \omega_{i}(\point),
\end{flalign}
where the maximum is taken over all tuples $1\leq k_{1}< \dotsm< k_{\nConsts}\leq n$ for which the denominator in the above expression does not vanish (which, again, is possible thanks to $\cmat$ being full rank). 
\typo{By \cref{asm:basic}\itemref{itm:bounded}, we have $K_{\start} \equiv \sup\setdef{\norm{\nabla\obj}_{\infty}}{\point\in\intpoints,\obj(\point)\leq \obj(\point^{0})} < \infty$ so}
$\omega_{i}(\point)$ is bounded in norm for all $i$ and all $\point\in\points$ and
\begin{flalign}
\norm{r(\point)}_{\infty}
	&\leq \norm{\nabla \obj(\point)}_{\infty} + \norm{\cmat^{\top}y(\point)}_{\infty}
	\leq \typo{K_{\start}} + \norm{\cmat}_{\ast}\norm{\omega(\point)}_{\infty},
\end{flalign}
where $\norm{\cmat}_{\ast} = \max_{1\leq i\leq n} \abs{\sum_{j=1}^{\nConsts}a_{ji}}$. 
This gives $\step_{0}(\point) \geq \eps / (K_{\start}+\norm{\cmat}_{\ast}\norm{\omega(\point)}) \geq \eps/K_{\start}$,
\typo{\ie $\inf\setdef{\step_{0}(\point)}{\point\in\intpoints,\obj(\point)\leq \obj(\point^{0})} > 0$}, as claimed.
\end{proof}

Our next result shows that
the Armijo step-size rule \eqref{eq:Armijo} terminates after finitely many iterations.

\begin{lemma}
\label{lem:Armijo}
Suppose that $\vecfield(\point)\neq0$, \ie $\point$ is not a \ac{KKT} point of $\obj$.
Then:
\begin{enumerate}
\item
The process \eqref{eq:Armijo} is well-defined at $\point$.
\item
$\step(\point)\geq\min\{2(1-\armijo)\beta\shrink/L,\underline\step(\point)\}$.
\end{enumerate}
\end{lemma}

Our proof builds on a classical line of reasoning as in \cite{AusTeb06}, but the algorithm's non-Euclidean nature necessitates some extra care:

\begin{proof}[Proof of \cref{lem:Armijo}]
Suppose that the Armijo backtracking process carries on without terminating at $\point\in\intpoints$.
Then, setting $\new\point(\lambda) = \point + \lambda \vecfield(\point)$ for all $\lambda>0$, and writing $\step \equiv \step(\point)$ and $\underline\step \equiv \underline\step(\point) = \min\{\step_{0}(\point),2\beta/L\}$ for concision, we get
\begin{flalign}
\obj(\new\point(\shrink^{\ell}\underline\step)) - \obj(\point)
	> \armijo\nabla\obj(\point)^{\top}(\new\point(\shrink^{\ell}\underline\step) - \point)
\end{flalign}
for all $\ell\in\N$.
Then, by the mean value theorem, there exists $\xi^{\ell}\in(\point,\new\point(\shrink^{\ell}\underline\step))$ such that 
\begin{flalign}
&\nabla\obj(\xi_{j}^{\typo{\ell}})^{\top} (\new\point(\shrink^{\ell}\underline\step) - \point)
	= \obj(\new\point(\shrink^{\ell}\underline\step)) - \obj(\point)
	> \armijo\nabla\obj(\point)^{\top} (\new\point(\shrink^{\ell}\underline\step) - \point).
\end{flalign}
Clearly, we also have $\xi^{\ell}\to \point$ as $\ell\to\infty$.
Hence, passing to the limit and recalling that $\armijo\in(0,1)$, we get 
\begin{flalign}
-\norm{\vecfield(\point)}_{\point}^{2}
	\geq -c\norm{\vecfield(\point)}_{\point}^{2}
	\iff \vecfield(\point)=0
	\iff \nabla\obj(\point)\in\scrA_{0}^{\bot},
\end{flalign}
a contradiction.

For our second claim, suppose that the Armijo criterion \eqref{eq:Armijo} is first satisfied at $\point$ after $\ell\geq1$ steps, \ie $\step/\shrink = \shrink^{\ell-1}\underline\step$.
By assumption, this means that we don't yet have sufficient decrease at the $(\ell-1)$-th step of the backtracking process, \ie
\begin{equation}
\label{eq:bound1}
\obj(\new\point(\step/\shrink))-\obj(\point)
	> \armijo\nabla\obj(\point)^{\top}(\new\point(\step/\shrink) - \point).
\end{equation}
Since $\nabla\obj$ is $L$-Lipschitz continuous relative to $\norm{\cdot}_{2}$, the descent inequality \eqref{eq:descent} for an arbitrary step-size $\lambda > 0$ becomes
\begin{flalign}
\obj(\new\point(\lambda))-\obj(\point)
	&\leq -\lambda\norm{\vecfield(\point)}_{\point}^{2}
	+ \frac{\lambda^{2}L}{2}\norm{\vecfield(\point)}_{2}^{2}
\end{flalign}
Thus, since $\norm{z}_{\point}^{2} = z^{\top} \hmat(\point)z \geq \lambda_{\min}(\hmat(z)) \norm{z}_{2}^{2} \geq \beta\norm{z}_{2}^{2}$ for all $z\in\R^{\nDims}$, we get 
\begin{flalign}
\obj(\new\point(\lambda))-\obj(\point)
	&\leq -\lambda\norm{\vecfield(\point)}_{\point}^{2}
	+ \frac{\lambda^{2}L}{2\beta}\norm{\vecfield(\point)}_{\point}^{2}
	= -\lambda \parens*{1 - \frac{\lambda L}{2\beta}} \norm{\vecfield(\point)}_{\point}^{2}
	\notag\\
	&= \parens*{1 - \frac{\lambda L}{2\beta}} \, \nabla\obj(\point)^{\top} (\new\point(\lambda) - \point),
\end{flalign}
where we used the angle condition \eqref{eq:angle} and the definition of $\new\point(\lambda)$.
Hence, setting $\lambda = \step/\shrink$, we get
\begin{equation}
\label{eq:bound2}
\obj(\new\point(\step/\shrink))-\obj(\point)
	\leq \left(1-\frac{L\step}{2\beta\shrink}\right) \nabla\obj(\point)^{\top}(\new\point(\step/\shrink) - \point)
\end{equation}
which, combined with \eqref{eq:bound1}, implies that $1 - \step L/(2\beta\shrink)  \leq \armijo$, \ie $\step \geq 2\beta\shrink (1-\armijo)/L$.

On the other hand, if the Armijo criterion \eqref{eq:Armijo} is already satisfied at $\point$ with step-size $\underline\step$ (\ie after $\ell=0$ shrinkage steps), we will have $\step = \underline\step$.
Thus, combining all of the above, we get $\step \geq \min\{\underline\step,2(1-\armijo)\beta\shrink/L\}$, as claimed.
\end{proof}

\typo{We are finally in a position to show that the algorithm's step-size is non-vanishing in the limit:

\begin{lemma}
\label{lem:step}
The algorithm's step-size sequence $\step^{\run}\equiv\step(\point^{\run})$ has $\inf_{\run}\step^{\run}>0$.
\end{lemma}

\begin{proof}[Proof of \cref{lem:step}]
\typo{Since $\obj(\point^{\run})$ is weakly decreasing (by the Armijo rule \eqref{eq:Armijo}), it follows that $\point^{\run} \in [\obj \leq \obj(\point^{0})]$ for all $\run$.
\cref{lem:lower} further guarantees that $\inf\setdef{\underline{\step}(\point)}{\points\in\intpoints,\obj(\point)\leq\obj(\point^{0})} > 0$, so our claim follows from \cref{lem:Armijo}.}
\end{proof}
} 

\subsection{Iterate analysis}
\label{sec:iterate-analysis}

We now turn to the long-run behavior of the iterates $\point^{\run}$ generated by \eqref{eq:HBA}.
The arguments are partly based on general facts on descent methods and extend the analysis of \cite{TseBomSch11} to a considerably richer algorithmic framework.
We start with a simple observation:

\begin{lemma}
\label{lem:iteratebounded}
Let $\point^{0}\in\intpoints$ be an initial condition satisfying \cref{asm:basic}\itemref{itm:bounded}.
Then the sequence of iterates $\point^{\run}$ of \eqref{eq:HBA} is bounded.
\end{lemma}

\begin{proof}
By the definition of \eqref{eq:HBA}, we have 
\begin{flalign}
\obj(\point^{k+1})
	\leq \obj(\point^{\run}) - \armijo\step^{\run}\norm{\vecfield(\point^{\run})}_{\point^{\run}}^{2},
\end{flalign}
showing that $\obj(\point^{\run})$ is non-increasing.
Our claim then follows trivially.
\end{proof}


The next result is actually a standard result for descent methods \textendash\ see \eg \cite{ABS13}:

\begin{lemma}
\label{lem:LimitSet}
With notation as in \cref{thm:main}, we have:
\begin{enumerate}
[\indent\upshape(\itshape a\hspace*{.5pt}\upshape)]
\item
The limit set $\scrL$ of \eqref{eq:HBA} is nonempty, compact and connected.
\item
$\lim_{k\to\infty} \dist(\point^{\run},\scrL)=0$.
\item
The objective function $\obj$ is constant on $\scrL$.
\end{enumerate}
\end{lemma}

With this lemma at hand, we proceed to show that the iterate change vanishes:

\begin{lemma}
\label{lem:changezero}
With notation as in \cref{thm:main}, we have $\lim_{k\to\infty}(\point^{k+1} - \point^{\run})=0$.
\end{lemma}

\begin{proof}
Observe that for all $k=0,1,\dotsc$, we have
\begin{flalign}
\norm{\vecfield(\point^{\run})}_{\point^{\run}}^{2}
	&= \frac{1}{(\step^{\run})^{2}}\norm{\hmat(\point^{\run})^{1/2} (\point^{k+1} - \point^{\run})}^{2}
	\geq\frac{\beta}{(\underline\step^{\run})^{2}}\norm{\point^{k+1} - \point^{\run}}^{2}.
\end{flalign}
Choose a convergent subsequence $\{\point^{\run}\}_{k\in\scrK}$, so that $\lim_{k\to\infty,k\in\scrK} \point^{\run}=\sol$.
Since $\obj(\point^{\run})$ is non-increasing, we readily get $\obj(\point^{\run}) \downarrow \obj(\sol) \leq \obj(\point^{0})$,
and also $\lim_{k\to\infty,k\in\scrK} [\obj(\point^{k+1})-\obj(\point^{\run})] = 0$.
Then, from \eqref{eq:Armijo}, it follows that
\(
\armijo\step^{\run} \norm{\vecfield(\point^{\run})}_{\point^{\run}}^{2}
	\leq \obj(\point^{\run}) - \obj(\point^{k+1})
\)
and hence, 
\(
\lim_{k\to\infty,k\in\scrK}\step^{\run} \norm{\vecfield(\point^{\run})}_{\point^{\run}}^{2} = 0.
\)
We thus get
\(
\limsup_{k\to\infty}\step^{\run}\norm{\vecfield(\point^{\run})}_{\point^{\run}}^{2}
	= \liminf_{k\to\infty}\step^{\run}\norm{\vecfield(\point^{\run})}_{\point^{\run}}^{2}
	=0.
\)
In turn, \cref{lem:Armijo} implies that $\inf_{k\in\N}\step^{\run}>0$, so $\lim_{k\to\infty}\norm{\vecfield(\point^{\run})}_{\point^{\run}}=0$.
\end{proof}


\begin{lemma}
\label{lem:KKTLimit}
$\scrL\subset\Lambda$.
\end{lemma}

\begin{proof}
Let $\{\point^{\run}\}_{k\in\N}$ be a convergent subsequence (we omit the relabeling).
Since $\vecfield(\point)=-\hmat(\point)^{-1}r(\point)$, we conclude from the above that 
\begin{flalign}
0
	&= \lim_{k\to\infty} \product{\hmat(\point^{\run})\vecfield(\point^{\run})}{\vecfield(\point^{\run})}
	= \lim_{k\to\infty} \norm{\hmat(\point^{\run})^{-1/2}r(\point^{\run})}^{2}.
\end{flalign}
Therefore, for all $i\in S$, we will have
\(
\lim_{k\to\infty} \abs{r_{i}(\point^{\run}) \theta''_{i}(\point^{\run}_{i})^{-1/2}}
	=0.
\)
Hence, if $i\in S_{\sol}$, we must have $\lim_{k\to\infty}r_{i}(\point^{\run})=r_{i}(\sol)=0$.
Now, for all $k$, the linear system 
\begin{equation}
\label{eq:rk}
(\nabla\obj(\point^{\run})-A^{\top}y)_{i}
	= r_{i}(\point^{\run}),
	\quad i\in S_{\sol},
\end{equation}
admits the solution $y^{\run}=y(\point^{\run})\in\R^{\nConsts}$.
Set $\dsol=y(\sol)=\lim_{k\to\infty}y(\point^{\run})$, by continuity.
Hence, passing to the limit in \cref{eq:rk} gives
\(
(\nabla\obj(\sol)-A^{\top}\dsol)_{i}
	= 0
\)
for all $i\in S_{\sol}$.
We thus conclude that
\(
\diag(\sol)\left(\nabla\obj(\sol) - A^{\top}\dsol\right)
	=0,
\)
\ie $\sol\in\Lambda$.
\end{proof}

\subsection{Proof of \cref{thm:main}}
We now combine all the above established preliminary facts, to prove the main results on the global convergence of \eqref{eq:HBA}.
Parts (a) and (b) of \cref{thm:main} follow from \cref{lem:iteratebounded,lem:KKTLimit}.
The remainder of this section is concerned with establishing claims \itemref{itm:cvx}\textendash\itemref{itm:strict} of \cref{thm:main}.
For this we have to show that $r(\sol)\geq 0$ for all $\sol\in\scrL$ holds under each of the conditions described in \cref{thm:main}.
The fact that $\sol$ is a \ac{KKT} point is then a consequence of \cref{lem:KKTLimit}, showing that also complementarity slackness holds.

\subsection*{Proof of \cref{thm:main}\itemref{itm:cvx}}
Assume that $\obj$ is convex.
Let $\sol\in\scrL$, and denote by $\bar{J} = \setdef{i\in S}{r_{i}(\sol)=0}$ and $\bar{J}^{c} = \setdef{i\in S}{r_{i}(\sol)\neq 0}$.
Moreover, define the set
\begin{equation}
\label{eq:Omega}
\Omega = \argmin\{\obj(\point): \point\in\points,\point_{\bar{J}^{c}}=0\}.
\end{equation}
Since $\obj$ is continuous and convex, the set $\Omega$ is closed and convex.
$\sol$ is a feasible point for the convex program \eqref{eq:Omega}, satisfying the \ac{KKT} condition $\diag(\sol)r(\sol)=0$ (\cref{lem:KKTLimit}).
Hence, $\Omega=\{\point\in\points: \obj(\point)=\obj(\sol),\point_{\bar{J}^{c}}=0\}$, and therefore $\obj$ is constant on $\Omega$.
By convexity, $\nabla\obj(\point)=\nabla\obj(\sol)$ for all $\point\in\Omega$.
We next prove that the reduced cost $r(\point)$ is constant on $\Omega$, and in fact must be non-negative, showing that $\sol\in\sols$.
\begin{lemma}
\label{lem:rconstant}
For all $\point\in\Omega$ we have $r(\point)=r(\sol)$.
\end{lemma}

\begin{proof}
Let $\point\in \Lambda$ be arbitrary.
We have 
\begin{flalign}
r(\point)
	&= \nabla\obj(\point)-A^{\top}y(\point)
	\notag\\
	&= \nabla\obj(\sol)-A^{\top}(A\hmat(\point)^{-1}A^{\top})^{-1}A\hmat(\point)^{-1}\nabla\obj(\point)
	\notag\\
	&= [\Id-A^{\top}(A\hmat(\point)^{-1}A^{\top})^{-1}A\hmat(\point)^{-1}]\nabla\obj(\sol)
	\notag\\
	&= [\Id-A^{\top}(A\hmat(\point)^{-1}A^{\top})^{-1}A\hmat(\point)^{-1}](r(\sol)+A^{\top}y(\sol))
	\notag\\
	&= r(\sol)-A^{\top}(A\hmat(\point)^{-1}A^{\top})^{-1}A\hmat(\point)^{-1}r(\sol)
	\notag\\
	&= r(\sol).
\end{flalign}
The first line is the definition of $r(\point)$, the second line is the definition of $y(\point)$ and uses the constancy of the gradient mapping on $\Lambda$.
The third line is then again the definition of $r(\sol)$.
In the last line we have used the fact that $\hmat(\point)^{-1}r(\sol) = (r_{j}(\sol)/\theta''(\sol_{j}))_{j\in S}=0$, which holds because if $i\in\bar{J}^{c}$ then $1/\theta''_{i}(\sol_{i})=0$, and the dual variable is bounded.
\end{proof}

We next prove that all accumulation points of \eqref{eq:HBA} are contained in $\Omega$.
To that end, for fixed $\eta>0$, we define 
\begin{flalign}
\Omega_{\eta} = \scrL\cap\{\point\in\Rn: \dist(\point,\Omega)<\eta\}.
\end{flalign}
Observe that this set is non-empty since $\sol\in\Omega$.
We will use this set to localize the limit points of the trajectory $\{\point^{\run}\}_{k\in\N}$.

\begin{lemma}
\label{lem:omega1}
If $\hat{\point}\in\scrL$ then $\hat{\point}\in\Omega$ or $\hat{\point}\notin\Omega_{\eta}$.
\end{lemma}

\begin{proof}
The proof follows via an argument by contradiction.
Assume that $\hat{\point}\notin\Omega$ and $\hat{\point}\in\Omega_{\eta}$.
Therefore, there exists a point $\tilde{\point}\in\Omega$ such that $\norm{\hat{\point}-\tilde{\point}}<\eta$.
Since $\obj(\sol)=\obj(\hat{\point})$ (\cref{lem:LimitSet}), there must exist $j\in\bar{J}^{c}$ such that $\hat{\point}_{j}>0$.
$r:\points\to\Rn$ is continuous and bounded.
$[-\infty<f\leq \obj(\point^{0})]$ is compact by assumption.
Hence, $r$ is uniformly continuous on $[-\infty<f\leq \obj(\point^{0})]$, guaranteeing the existence of a scalar $\eta>0$ such that 
\begin{equation}
\norm{r(\point)-r(z)}
	\leq \min_{i\in \bar{J}^{c}} r_{i}(\sol)/2
\end{equation}
whenever $\obj(\point), \obj(z) \leq \obj(\point^{0})$ and $\norm{\point-z}\leq\eta$.
In particular, the uniform continuity of the dual variable guarantees that 
\begin{flalign}
\abs{r_{j}(\tilde{\point})-r_{j}(\hat{\point})}
	\leq \abs{r_{j}(\sol)}/2,
\end{flalign}
for some $j\in\bar{J}^{c}$.
Since $r_{j}(\tilde{\point})=r_{j}(\sol)$ by \cref{lem:rconstant}, this implies $r_{j}(\hat{\point})\geq \abs{r_{j}(\sol)}/2 > 0$.
Hence $\diag(\hat{\point})r(\hat{\point})\neq0$, contradicting the conclusion $\hat{\point}\in\scrL\subset\Lambda$ of \cref{lem:KKTLimit}.
\end{proof}

\begin{lemma}
\label{lem:omega2}
$\scrL\subseteq\Omega$.
\end{lemma}

\begin{proof}
Assume there exists an accumulation point $\hat{\point}\notin\Omega$.
From \cref{lem:omega1}, we deduce that $\hat{\point}\notin\Omega_{\eta}$.
Since the limit set is connected, it follows that the sequence $\{\point^{\run}\}_{k\in\N}$ must have accumulation points in $\Omega_{\eta}\setminus\Omega$.
Hence, there exists $\tilde{\point}\in\scrL\cap(\Omega_{\eta}\setminus\Omega)$.
In particular, $\tilde{\point}\in\scrL$, so that $\obj(\tilde{\point})=\obj(\sol)$.
Furthermore, $\tilde{\point}\notin\Omega$, so there exists $j\in\bar{J}^{c}$ such that $\tilde{\point}_{j}>0$.
From this we derive the same contradiction as in \cref{lem:omega1}.
\end{proof}
This shows that for every converging subsequence $\point^{k_{q}}$, we have $\lim_{q\to\infty}r(\point^{k_{q}})=r(\sol)$.
Suppose now that $r_{j}(\sol)\equiv \bar{r}_{j}<0$ for some $j\in S$.
Then, by the complementarity condition $\diag(\sol)r(\sol)=0$, we have $j\in\bar{J}^{c}$.
By continuity, we know that there exists a $\kappa\in\N$ such that $r_{j}(\point^{\run})<0$ for all $k$ far along the subsequence, say all $k\geq\kappa$.
Therefore, for all $k\geq\kappa$ we conclude 
\begin{flalign}
\point^{k+1}_{j}
	&= \point_{j}^{\run} - \step^{\run} r_{j}(\point^{\run})/\theta''_{j}(\point_{j}^{\run})
	> \point_{j}^{\run}.
\end{flalign}
By induction, we conclude that $\point_{j}^{\run}> \point_{j}^{\kappa}\geq 0$ for all $k\geq\kappa$, a contradiction.
\cref{thm:main}\itemref{itm:cvx} now follows from the \ac{KKT} conditions \eqref{eq:KKT1} and \eqref{eq:KKT2}.
\qed

\subsection*{Proof of \cref{thm:main}\itemref{itm:isolated}}
We know that $\scrL$ is a connected set.
From \cref{lem:KKTLimit}, we know that $\scrL\subset\Lambda$.
Since the iterate changes goes to zero (\cref{lem:changezero}), this implies that the entire sequence converges.
Hence, $\scrL=\{\sol\}\in\points$, with $\sol$ depending only on the initial condition.
Since $\diag(\sol)r(\sol)=0$ by complementarity, the same contradiction argument used in the previous paragraph rules out the possibility that $r_{i}(\sol)<0$ for some $i\in\bar{J}^{c}$.
Hence, $\sol$ is a \ac{KKT} point and our claim follows. 
\qed

\subsection*{Proof of \cref{thm:main}\itemref{itm:strict}}
Let $\sol\in\scrL$ and let $\bar{J}_{0} = \{i\in S: r_{i}(\sol)=0\},\bar{J}_{+}=\{i\in S: r_{i}(\sol)>0\},\bar{J}_{-}=\{i\in S: r_{i}(\sol)<0\}$.
Now, define the set 
\begin{flalign}
\bar{\Lambda}
	= \setdef{\point\in\Lambda}{r_{\bar{J}_{0}}(\point)=0,r_{\bar{J}_{+}}(\point)>0,r_{\bar{J}_{-}}(\point)<0}
\end{flalign}
and let $\scrB = \setdef{\point\in\Rn}{\norm{\point}\leq 1}$ be the unit ball in $\Rn$.
By the primal non-degeneracy assumption and strict complementarity, $\bar{\Lambda}$ is isolated from the rest of $\Lambda$.
Hence, there exists $\shrink>0$ such that $(\bar{\Lambda}+\shrink\scrB)\cap\Lambda=\bar{\Lambda}$.
Since $\scrL$ is connected and contained in $\Lambda$, we conclude that
\(
\scrL\cap(\bar{\Lambda} + \shrink\scrB)
	\subseteq \Lambda\cap(\bar{\Lambda}+\shrink\scrB)=\bar{\Lambda}.
\)
Hence, for every $j\in\bar{J}_{-}$ we have $r_{j}(\point^{\run})<0$ for all $k$ sufficiently large.
Repeating the argument we used to prove part \itemref{itm:cvx} of the theorem, we again arrive at a contradiction.
We conclude that $\bar{J}_{-}=\emptyset$, \ie $r(\sol)\geq0$.
\qed

\section{Convergence rate}
\label{sec:rates}

\newcommand{\objdiff}{d}

In this section, we establish an estimate of the value convergence rate of \eqref{eq:HBA} in the special case where $\obj$ is quadratic, \ie
\begin{equation}
\label{eq:quadratic}
\obj(x)
	= \frac{1}{2}x^{\top}Qx
	+ \cost^{\top}x
\end{equation}
for some symmetric $Q\in\scrS^{\nDims}$ and $\cost\in\Rn$.
When $Q$ is the zero matrix, we recover a linear programming problem.
In the rest of this section, we will focus on the challenging case where $Q$ has at least one negative eigenvalue, in which case \eqref{eq:Opt} is NP-complete \cite{Vav90}.

Our proof establishes sublinear convergence of the sequence $\obj(x^{k})$ to a \ac{KKT} point.
This result generalizes and extends previous work of Tseng \cite{Tse04} and Tseng et al. \cite{TseBomSch11}.
Our results are based on techniques developed by \cite{TseBomSch11};
however, the introduction of a Riemannian metric necessitates a series of intricate estimates in order to establish a rate of convergence.
Specifically, our analysis requires some mild additional control on the metric-inducing kernels close to the boundary of the feasible set, which we call \emph{moderate steepness}:


\begin{assumption}
\label{asm:moderate}
A kernel function $\theta\from(0,\infty)\to\R$ is \emph{moderately steep} at $0$ if there exist some $\eps_{i}\in(0,1)$, $\omega\geq1/2$ and $m,M>0$ such that
\begin{equation}
\label{eq:moderate}
\frac{m}{s}
	\leq\theta''(s)
	\leq \frac{M}{s^{2\omega}}
	\quad
	\text{for all $s\in(0,\eps)$}.
\end{equation}
\end{assumption}

We verify below that the kernels described in \cref{sec:HR} satisfy this condition:
\smallskip
\begin{enumerate}
\addtolength{\itemsep}{\smallskipamount}
\item
$\theta(t) = \frac{1}{2} \beta t^{2} + t\log t$ for $t\geq 0$.
Then $\theta'(t) = \beta + 1/t$, and \eqref{eq:moderate} is satisfied with $\omega = 1/2$, $m=1$ and $M=1+\beta\eps$.
\item
$\theta(t)= \frac{1}{2}\beta t^{2} + \frac{1}{(1-\pexp)(2-\pexp)} t^{2-\pexp}$, $\pexp\in(1,2)$.
Then $\theta''(t) = \beta + 1/t^{\pexp}$, so \eqref{eq:moderate} is satisfied with $m = p\eps^{p-1}$, $M=\beta\eps^{2\omega}+p\eps^{p+2(\omega-1)}$, and $\omega=1$.
\item
$\theta(t) = \frac{1}{2}\beta t^{2} - \log t$.
Then $\theta''(t) = \beta + \frac{1}{t^{2}}$, and \eqref{eq:moderate} is satisfied with $m=\frac{1}{\eps}$ and $M=\beta\eps^{2\omega} + \eps^{2(\omega-1)}$, and $\omega=1$.
\end{enumerate}

Under the assumption that all the metric-inducing kernels satisfy the moderate steepness property, we are able to obtain the announced sublinear convergence rate of the function value sequence.
\begin{theorem}
\label{thm:rate}
Assume $\obj$ is of the form \eqref{eq:quadratic} for some $Q\in\R^{n\times n}$, $\cost\in\Rn$.
Suppose that \eqref{eq:HBA} is run with metric-inducing kernels $\theta_{1}(x),\ldots,\theta_{n}(x),$ satisfying \cref{asm:moderate}, and generating the sequence $(x^{k})_{k\geq 0}$.
Then $\obj(x^{k})$ converges to some $\obj_{\infty}\in\R$ and
\begin{equation}
\obj(x^{k}) - \obj_{\infty}
	= \bigoh(k^{-\rho})
\end{equation}
where $\bar\omega = \max\{1,\omega\}$ and $\rho = 1/(2\bar\omega-1)$.
\end{theorem}
\begin{proof}
Let $r^{k}\equiv r(x^{k})$, $y^{k}\equiv y(x^{k})$, and set $\eta^{k}:=H(x^{k})^{1/2}v(x^{k})=-H(x^{k})^{-1/2}r^{k}$.
Since $\lim_{k\to\infty}(\obj(x^{k+1}) - \obj(x^{k}))=0$, and Armijo backtracking guarantees sufficient decrease by
\begin{flalign}
\obj(x^{k+1})
	\leq \obj(x^{k}) + \armijo\step^{k}\nabla \obj(x^{k})^{\top}v(x^{k})
	= \obj(x^{k}) - \armijo\step^{k}\norm{H(x^{k})^{1/2}v(x^{k})}^{2}_{2},
\end{flalign}
it follows that $\eta^{k}\to 0$.
For $J\in 2^{S}$, define
\begin{equation}
\label{eq:K}
\scrK_{J}
	=
	\setdef{k\in\N_{0}}{\theta_{j}''(x_{j}^{k})^{-1/2} \leq \abs{\eta_{j}^{k}}^{1/2}\;\forall j\in J\text{ and }\abs{r_{j}^{k}} \leq \abs{\eta_{j}^{k}}^{1/2}\;\forall j\in J^{c} }.
\end{equation}
Since $\abs{\eta_{j}^{k}} = \abs{r_{j}^{k} \theta_{j}''(x_{j}^{k})^{-1/2}}$ by definition, it follows that either $\abs{r_{j}^{k}} \leq \abs{\eta_{j}^{k}}^{1/2}$ or $\theta_{j}''(x_{j}^{k})^{-1/2} \leq \abs{\eta_{j}^{k}}^{1/2}$.
Hence, for every $k\in\N_{0}$, there exists at least one $J\in 2^{S}$ such that $k\in \scrK_{J}$.
Since $2^{S}$ is finite, there is at least one set $J$ for which $\scrK_{J}$ is infinite.
Fix such a set $J$.
For all $k\in\scrK_{J}$, consider the system of linear inequalities defining a point $(p,z)\in\R^{n}\times\R^{m}\cong\R^{n+m}$, given by
\begin{flalign}
p_{J}
	&= x_{j}^{k},
	\quad
	q_{j}^{\top}p-a_{j}^{\top}z=-c_{j} + r_{j}^{k}
	\quad
	\text{for all $j\in J^{c}\equiv S\setminus J$},
	\notag\\
p
	&\geq 0,
	\quad
	Ap=b.
\end{flalign}
Let $\scrP_{k}$ be the polyhedron defined by these inequalities.
Since $(x^{k},y^{k})$ satisfies these inequalities, we have $\scrP_{k}\neq\emptyset$ for all $k\in\scrK_{J}$.
Moreover, for all $j\in J$, we have $\lim_{k\to\infty,k\in\scrK_{J}} \theta_{j}''(x_{j}^{k}) = \infty$, implying in turn that $\lim_{k\to\infty,k\in\scrK_{J}}x_{j}^{k} = 0$.
Therefore, for all $k\in\scrK_{J}$ sufficiently large, \cref{asm:moderate} yields for $M_{\ast} = \max\{M_{1},\dotsc,M_{n}\}$ the bound
\begin{subequations}
\begin{flalign}
(x_{j}^{k})^{\omega}
	&\leq M_{j}^{1/2} \abs{\eta_{j}^{k}}^{1/2}
	\leq M_{\ast}^{1/2} \abs{\eta_{j}^{k}}^{1/2}
	\quad
	\text{for all $j\in J$},
	\\
\abs{r_{j}^{k}}
	&\leq \abs{\eta_{j}^{k}}^{1/2}
	\quad
	\text{for all $j\in J^{c}$}.
\end{flalign}
\end{subequations}
If $\omega\in[1/2,1)$, then $\abs{\eta_{j}^{k}}^{\frac{1}{2\omega}} \leq \abs{\eta_{j}^{k}}^{\frac{1}{2}}$, and therefore
\begin{subequations}
\begin{alignat}{2}
& x_{j}^{k}
	\leq C_{1}^{1/2} \abs{\eta_{j}^{k}}^{1/2}
	\quad
	&&\text{for all $j\in J$},
\label{eq:xomegasmall}
\\
\abs{r_{j}^{k}}
	&\leq C_{1}^{1/2} \abs{\eta_{j}^{k}}^{1/2}
	\quad
	&&\text{for all $j\in J^{c}$},
\label{eq:romegasmall}
\end{alignat}
\end{subequations}
for all $k\in\scrK_{J}$ sufficiently large, where we set $C_{1} = \max\{1,M_{\ast},M_{\ast}^{1/\omega}\}$.

If $\omega\geq 1$, then $\abs{\eta_{j}^{k}}^{1/(2\omega)}\geq \abs{\eta_{j}^{k}}^{1/2}$, and therefore, for $k\in\scrK_{J}$ sufficiently large, we get
\begin{subequations}
\begin{alignat}{2}
x_{j}^{k}
	&\leq C_{1}^{1/2}\abs{\eta_{j}^{k}}^{1/(2\omega)}
	\quad
	&&\text{for all $j\in J$},
	\\
\abs{r_{j}^{k}}
	&\leq C_{1}^{1/2} \abs{\eta_{j}^{k}}^{\frac{1}{2\omega}}
	\quad
	&&\text{for all $j\in J^{c}$}.
\end{alignat}
\end{subequations}
Setting $\bar\omega = \max\{1,\omega\}$, the previous two estimates yield
\begin{subequations}
\begin{alignat}{2}
x_{j}^{k}
	&\leq C_{1}^{1/2}\abs{\eta_{j}^{k}}^{1/(2\bar\omega)}
	\quad
	&&\text{for all $j\in J$},
	\\
\abs{r_{j}^{k}}
	&\leq C_{1}^{\frac{1}{2}} \abs{\eta_{j}^{k}}^{1/(2\bar\omega)}
	\quad
	&&\text{for all $j\in J^{c}$},
\label{eq:xomegabig}
\end{alignat}
\end{subequations}
and hence
\begin{equation}
\label{eq:boundxk}
\norm{(x_{j}^{k},r^{k}_{J^{c}})}_{2\bar\omega}^{2\bar\omega} \leq C_{1}^{\bar\omega}\norm{\eta^{k}}_{1},
\end{equation}
for all $k\in\scrK_{J}$ sufficiently large.

Hence, since $\eta^{k}\to 0$, we see that $\{(x_{j}^{k},r^{k}_{J^{c}})\}_{k\in\scrK_{J}}\to 0$.
This implies that the right-hand side defining the polyhedron $\scrP_{k}$ is uniformly bounded.
Since $\{(x^{k},y^{k})\}_{k\in\scrK_{J}}$ is bounded (see the proof of \cref{lem:lower}), any cluster point of this sequence must satisfy
\begin{subequations}
\begin{flalign}
p_{J}
	&= 0,
	\quad
	q^{\top}_{j}p-a^{\top}_{j}z=-c_{j}
	\quad
	\text{for all $j\in J^{c}$},
	\\
p
	&\geq 0
	\quad
	Ap=b.
\end{flalign}
\end{subequations}
Call $\scrP_{J}$ the polyhedron defined by the above linear inequalities.
Let $(\bar{x}^{k},\bar{y}^{k})$ denote the Euclidean projection of $(x^{k},y^{k})$ onto $\scrP_{J}$.
Since $(x^{k},y^{k})\in\scrP_{k}$ for all $k\in\scrK_{J}$, Hoffman's error bound \cite[Corollary 3.2.5]{FacPan03} implies that
\begin{equation}
\norm{(\bar{x}^{k},\bar{y}^{k})-(x^{k},y^{k})}_{2} \leq C_{2}\norm{(x_{j}^{k},r^{k}_{J^{c}})}_{2\bar\omega}\qquad\forall k\in\scrK_{J},
\end{equation}
where $C_{2}$ is a constant that depends only on $\bar\omega,Q,A$ and $J$.
Combining this with eq.
\eqref{eq:boundxk} shows that
\begin{equation}
\norm{(\bar{x}^{k},\bar{y}^{k})-(x^{k},y^{k})}_{2} \leq C_{2}C_{1}^{\frac{1}{2}}\norm{\eta^{k}}^{\frac{1}{2\bar\omega}}_{1}\qquad\forall k\in\scrK_{J} 
\textup{ sufficiently large}.
\end{equation}
We next claim that $\obj$ is constant on $\scrP_{J}$.
To see this, let $(p,z),(p',z')\in\scrP_{J}$ arbitrary.
Then,
\begin{flalign}
\obj(p) - \obj(p')
	&= \frac{1}{2}(p-p')^{\top}Q(p-p')+(c+Qp')^{\top}(p-p')
	\notag\\
	&= \frac{1}{2}(p-p')^{\top}Q(p-p')+(c+Qp'-Az')^{\top}(p-p')
	\notag\\
	&= \frac{1}{2}(p-p')^{\top}Q(p-p')
\end{flalign}
where the second equality follows from the fact that $A(p-p')=0$, and the third equality follows from the definition of $\scrP_{J}$.
Similarly $\obj(p') - \obj(p) = \frac{1}{2}(p-p')^{\top}Q(p-p')$, resulting in $\obj(p')=\obj(p)$.

Next, observe that
\begin{flalign}
(Q\bar{x}^{k}+c)^{\top}(x^{k}-\bar{x}^{k})
	&= (Q\bar{x}^{k}+c-A^{\top}\bar{y}^{k})^{\top}(x^{k}-\bar{x}^{k})
	\notag\\
	&= \sum_{j\in J}(q^{\top}_{j}\bar{x}^{k}+c_{j}-a_{j}^{\top}\bar{y}^{k})x_{j}^{k}
	\notag\\
	&= \sum_{j\in J}(q^{\top}_{j}(\bar{x}^{k}-x^{k})-a_{j}^{\top}(\bar{y}^{k}-y^{k})-r_{j}^{k})x_{j}^{k}.
\end{flalign}
From this, we compute
\begin{flalign}
\abs{\obj(x^{k}) - \obj(\bar{x}^{k})}
	&= \abs{\frac{1}{2}(x^{k}-\bar{x}^{k})^{\top}Q(x^{k}-\bar{x}^{k})+(Q\bar{x}^{k}+c)^{\top}(x^{k}-\bar{x}^{k})}
	\notag\\
	&\leq \frac{1}{2} \lambda_{\max}(Q) \, \norm{x^{k}-\bar{x}^{k}}^{2}_{2}
	+ \abs*{\sum_{j\in J}q_{j}^{\top}(\bar{x}^{k}-x^{k})-a_{j}^{\top}(\bar{y}^{k}-y^{k})+r_{j}^{k})x_{j}^{k}}
	\notag\\
	&\leq \frac{1}{2} \lambda_{\max}(Q) \, \norm{x^{k}-\bar{x}^{k}}^{2}_{2}
	\notag\\
	&\quad
	+ \sum_{j\in J} \bracks*{\norm{(q_{j},-a_{j})}
		\cdot
		\norm{(\bar{x}^{k},\bar{y}^{k})-(x^{k},y^{k})}_{2}x_{j}^{k}+x_{j}^{k}\abs{r_{j}^{k}}}
\end{flalign}
Collecting all the information from the previous estimates, we can bound each of these terms for $k\in\scrK_{J}$ sufficiently large and $j\in J$ , as follows:
\begin{itemize}
\item
$\norm{x^{k}-\bar{x}^{k}}_{2}^{2} \leq\norm{(x^{k},y^{k})-(\bar{x}^{k},\bar{y}^{k})}_{2}^{2} \leq C_{2}^{2}C_{1}\norm{\eta^{k}}_{1}^{1/\bar\omega}.$
\item
$x_{j}^{k} \leq C_{1}^{1/2} \abs{\eta_{j}^{k}}^{1/(2\bar\omega)}$.
\item
$x_{j}^{k}\abs{r_{j}^{k}} \leq C_{3}\abs{\eta_{j}^{k}}^{1/\bar\omega}$.
\end{itemize}
To see the last relation, observe that if $\omega\in[1/2,1)$, we have
\begin{flalign}
x_{j}^{k}\abs{r_{j}^{k}}
	&= x_{j}^{k}\abs{\eta_{j}^{k}} \theta_{j}''(x_{j}^{k})^{1/2}
	\leq C_{1}^{1/2} (x_{j}^{k})^{1-\omega} \abs{\eta_{j}^{k}}
	\leq C_{1}^{1 - \omega/2} \abs{\eta_{j}^{k}}^{(3-\omega)/2}
	\leq C_{1}\abs{\eta_{j}^{k}}.
\end{flalign}
The first equality uses the identity $r_{j}^{k}=-\eta_{j}^{k}\theta_{j}''(x_{j}^{k})^{1/2}$.
The first inequality uses \cref{asm:moderate}, and the second inequality is a consequence of relation \eqref{eq:xomegasmall}.
The final inequality follows since $\eta_{j}^{k}\to 0$ as $\scrK_{J}\ni k\to\infty.$ Now assume that $\omega\geq 1$.
We first deduce from \cref{asm:moderate} the inequality
$(x_{j}^{k})^\omega\abs{r_{j}^{k}} \leq C_{1}^{1/2}\abs{\eta_{j}^{k}}$, and then
\begin{equation}
\left(x_{j}^{k}\abs{r_{j}^{k}}\right)^\omega
	\leq C_{4} (x_{j}^{k})^\omega\abs{r_{j}^{k}}
	\leq C_{4} C_{1}^{1/2} \abs{\eta_{j}^{k}},
\end{equation}
where $C_{4} = \max_{k\ge1}\abs{r_{j}^{k}}^{\omega-1}<\infty$.
Departing from this relation, we obtain $x_{j}^{k}\abs{r_{j}^{k}} \leq C_{4}^{1/\omega} C_{1}^{1/(2\omega)} \abs{\eta_{j}^{k}}^{1/\omega}$.
To combine the two cases, set $C_{3}:=\max\{C_{1},C_{4}^{1/\omega} C_{1}^{1/(2\omega)}\},$ and recall that $\bar\omega=\max\{1,\omega\}$.
\\

Using all these bounds, we conclude that there exists a constant $C_{J}>0$, such that
\begin{flalign}
\abs{\obj(x^{k}) - \obj(\bar{x}^{k})} \leq C_{J}\norm{\eta^{k}}_{1}^{1/\bar\omega}.
\end{flalign}
for all $k\in\scrK_{J}$ sufficiently large.
Let $C_{\ast}$ be the maximum of $C_{J}$ over all $J\in 2^{S}$ for which $\scrK_{J}$ is infinite.
Thus, there exists an index $\bar{k}\in\N$ sufficiently large, so that for all $k\geq \bar{k}$ we have
\begin{equation}
\label{eq:eta1}
\abs{\obj(x^{k}) - \obj(\bar{x}^{k})}
	\leq C_{\ast}\norm{\eta^{k}}_{1}^{\frac{1}{\bar\omega}}.
\end{equation}
The sequence $\obj(x^{k})$ is bounded and decreasing, so there exists $\obj_{\infty}>-\infty$ such that $\obj(x^{k})\downarrow \obj_{\infty}$.
Since $\{\bar{x}^{k}\}_{k\in\scrK_{J}}\subset\scrP_{J}$, it follows from the constancy of $\obj$ on $\scrP_{J}$ that $\obj(\bar{x}^{k})=\obj_{\infty}$ for all $k\in\scrK_{J}$, and thus for all $k \geq \bar k$.
Hence, \eqref{eq:eta1} becomes
\begin{equation}
\label{eq:basicdelta}
\obj(x^{k})-\obj_{\infty}\leq C_{\ast}\norm{\eta^{k}}_{1}^{1/\bar\omega}\qquad\forall k\geq\bar{k}.
\end{equation}
Set $\objdiff^{k}:=\obj(x^{k})-\obj_{\infty}$.
Armijo backtracking then gives
\begin{flalign}
\obj(x^{k+1}) - \obj(x^{k})
	&\leq -\armijo\step^{k}\norm{\eta^{k}}^{2}_{2}
	\leq -\armijo\step^{k}C_{5}\norm{\eta^{k}}_{1}^{2}
	\leq -C_{6}\norm{\eta^{k}}^{2}_{1}.
\end{flalign}
Here the constant $C_{5}$ captures the equivalence of the norms $\norm{\cdot}_{1}$ and $\norm{\cdot}_{2}$, and the constant $C_{6}$ incorporates the boundedness of the step size sequence $\{\step^{k}\}_{k}$.
Hence,
\begin{flalign}
\obj(x^{k}) - \obj(x^{k+1})=\objdiff^{k}-\objdiff^{k+1}\geq C_{6}\norm{\eta^{k}}^{2}_{1}.
\end{flalign}
Combining this with \eqref{eq:basicdelta}, we conclude that
\begin{flalign}
C_{6}^{-1/(2\bar\omega)} (\objdiff^{k}-\objdiff^{k+1})^{1/(2\bar\omega)}
	\geq \norm{\eta^{k}}_{1}^{1/\bar\omega}
	\geq \objdiff^{k}/C_{\ast}.
\end{flalign}
Hence, for $\kappa = C_{\ast} / C_{6}^{1/(2\bar\omega)}$, we obtain the recursion
\begin{equation}
\kappa\left(\objdiff^{k}-\objdiff^{k+1}\right)^{1/(2\bar\omega)}
	\geq \objdiff^{k}.
\end{equation}
This can be rearranged to yield the equivalent expression
\begin{flalign}
\objdiff^{k+1}
	\leq\objdiff^{k} - \left(\objdiff^{k}/\kappa\right)^{2\bar\omega}
\end{flalign}
for $k \geq \bar k$.
Now, write $\phi(\objdiff^{k})$ for the RHS of the above, and observe that the function $\phi$ is strictly increasing on the interval $[0,\tilde x]$, where  $\tilde x = \left(\kappa^{2\bar\omega}/(2\bar\omega)\right)^{1/(2\bar\omega-1)}$.
Then, fix $\rho=1/(2\bar\omega-1)\in(0,1]$ and choose constants $C>0$ and $K\in\N,K \geq \bar k$ such that $C \geq \kappa^{2\bar\omega/(2\bar\omega-1)}$ 
and $\objdiff^K\leq CK^{-\rho} \leq \tilde x$.

Such a choice of constants $C,K$ is indeed possible:
First look for $K \geq \bar k$ such that $\objdiff^{k} \leq \tilde x$ for all $k \geq K$ and choose $C = \kappa^{2\bar\omega/(2\bar\omega - 1)}$.
If $\objdiff^K\leq CK^{-\rho} \leq \tilde x$ holds, there is nothing further to do.
If $\objdiff^K> CK^{-\rho}$, increase $C$ such that $CK^{-\rho}=\tilde x$ holds.
If $CK^{-\rho}>\tilde x$, increase $K$ to achieve $CK^{-\rho} \leq \tilde x$ and then again increase $C$ such that $CK^{-\rho}=\tilde x$ holds.
Then, $\left(C/\kappa\right)^{2\bar\omega} \geq C$, and therefore we have
\begin{equation}
\left(\frac{C}{k^{\rho}\kappa}\right)^{2\bar\omega}
	\geq Ck^{-2\rho\bar\omega}.
\end{equation}

We will now prove by induction the claim that $\objdiff^{k} \leq Ck^{-\rho} \leq \tilde x$ holds for all $k\ge K$.
The base case $k=K$ holds by construction of $C$ and $K$.
Assume now $k\ge K$ and $\objdiff^{k} \leq Ck^{-\rho} \leq \tilde x$.
Then we obtain
\begin{align}
\objdiff^{k+1}
	&\leq \objdiff^{k} - \left(\objdiff^{k}/\kappa\right)^{2\bar\omega}= \phi(\objdiff^{k})\leq \phi(Ck^{-\rho})
	\leq Ck^{-\rho} - Ck^{-2\rho\bar\omega}
	\leq \frac{C}{(k+1)^{\rho}},
\end{align}
where we used the fact that $\rho\leq 1$ and $1 - 1/k \leq (1 - 1/k)^{\rho}$.
This shows that
\begin{equation}
\obj(x^{k})-\obj_{\infty}
	\leq Ck^{-\rho}\qquad\forall k\geq K,
\end{equation}
so our proof is complete.
\end{proof}

\section{Numerical Experiments}
\label{sec:numerics}
%

In this section, we validate the theoretical analysis of the previous sections via a series of numerical experiments and practical applications.

\subsection{Experiments with common benchmarks}
\label{sec:benchmarks}


\begin{figure}[tbp]
\centering
\begin{subfigure}{\textwidth}
\includegraphics[width=.49\textwidth]{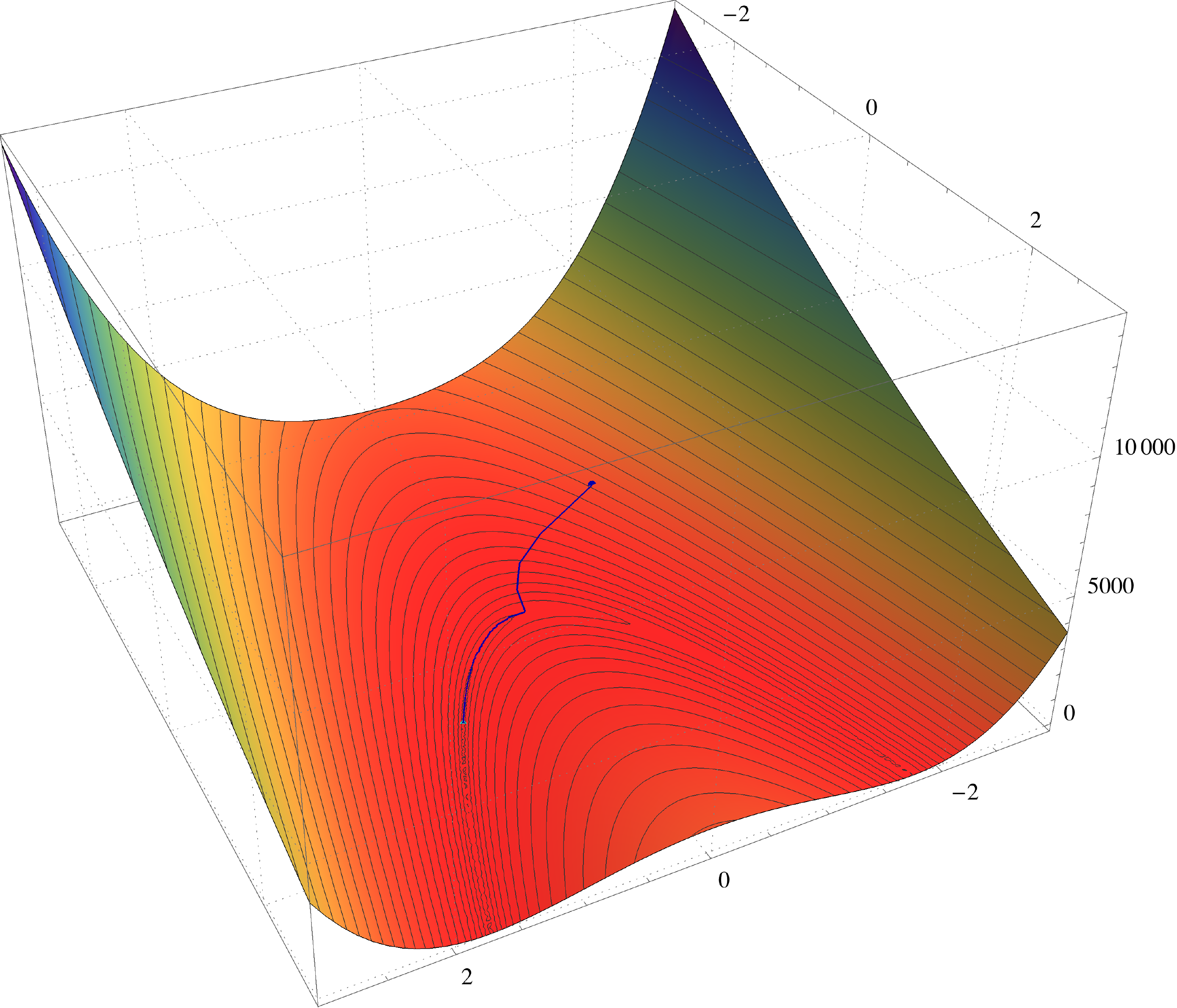}
\hfill
\includegraphics[width=.49\textwidth]{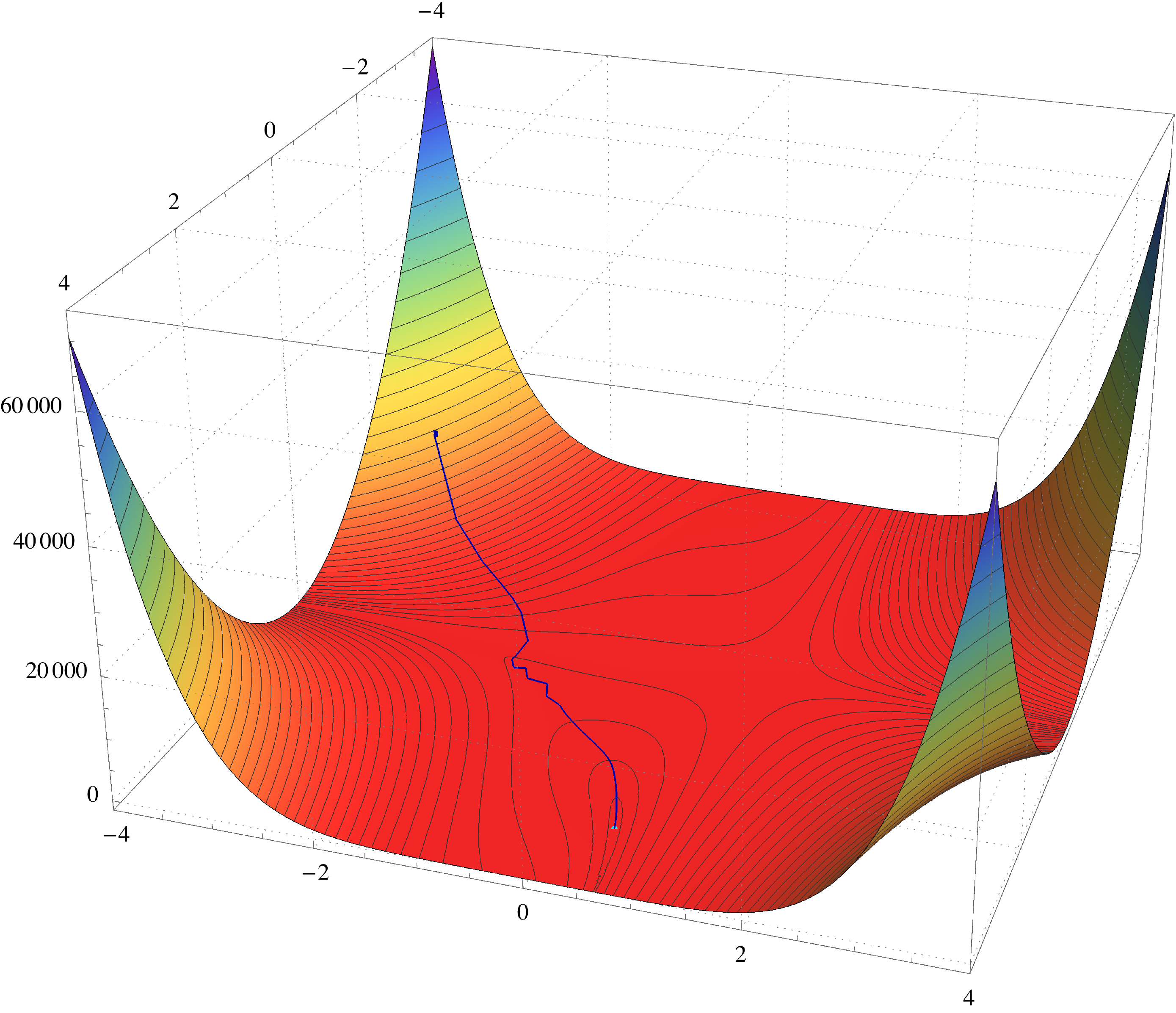}
\caption{\ac{HBA} trajectries for the Rosenbrock and Beale functions (left and right respectively).}
\label{fig:test-trajectories}
\vspace{1ex}
\end{subfigure}
\begin{subfigure}{\textwidth}
\includegraphics[width=.48\textwidth]{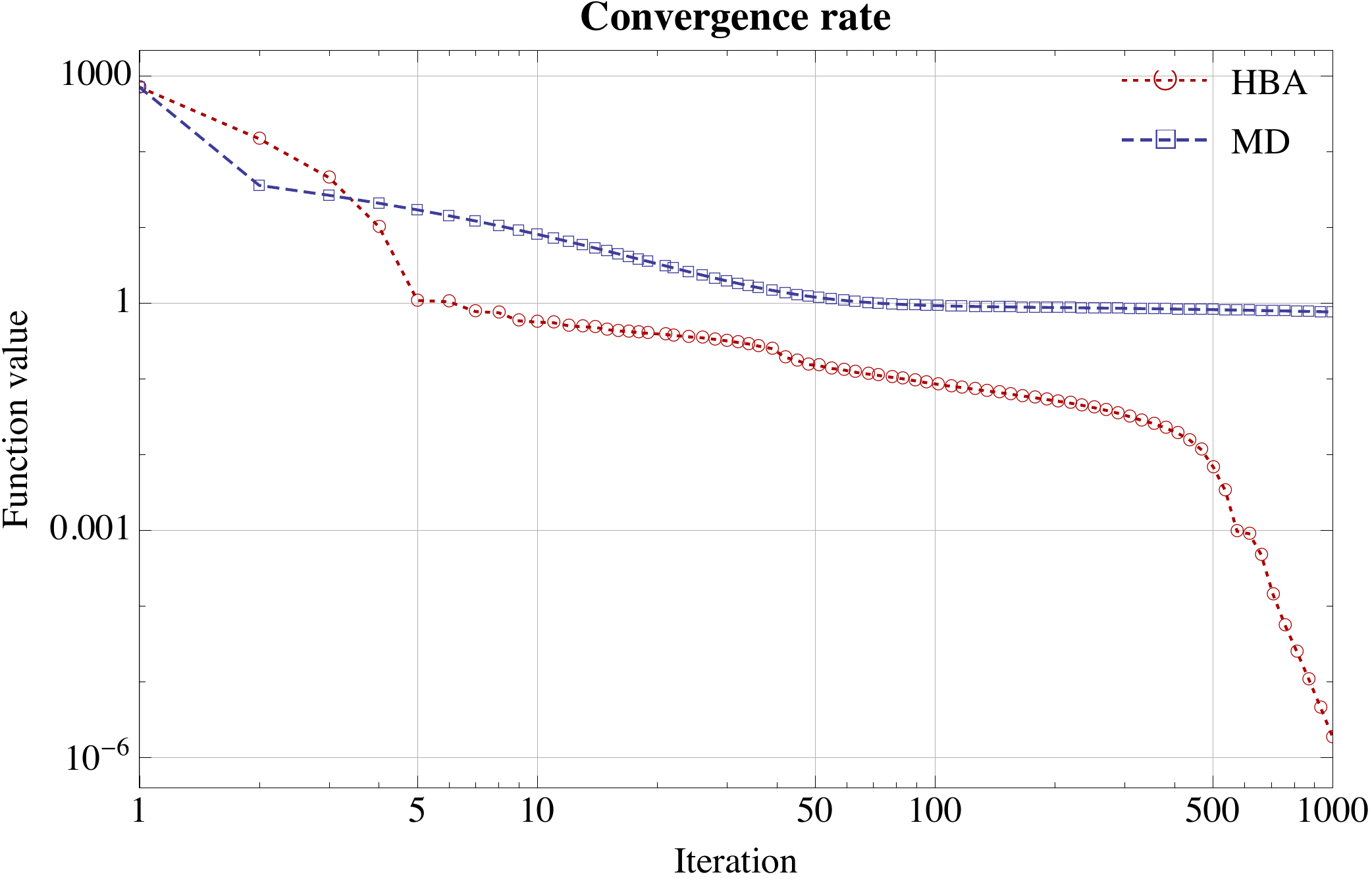}
\hfill
\includegraphics[width=.48\textwidth]{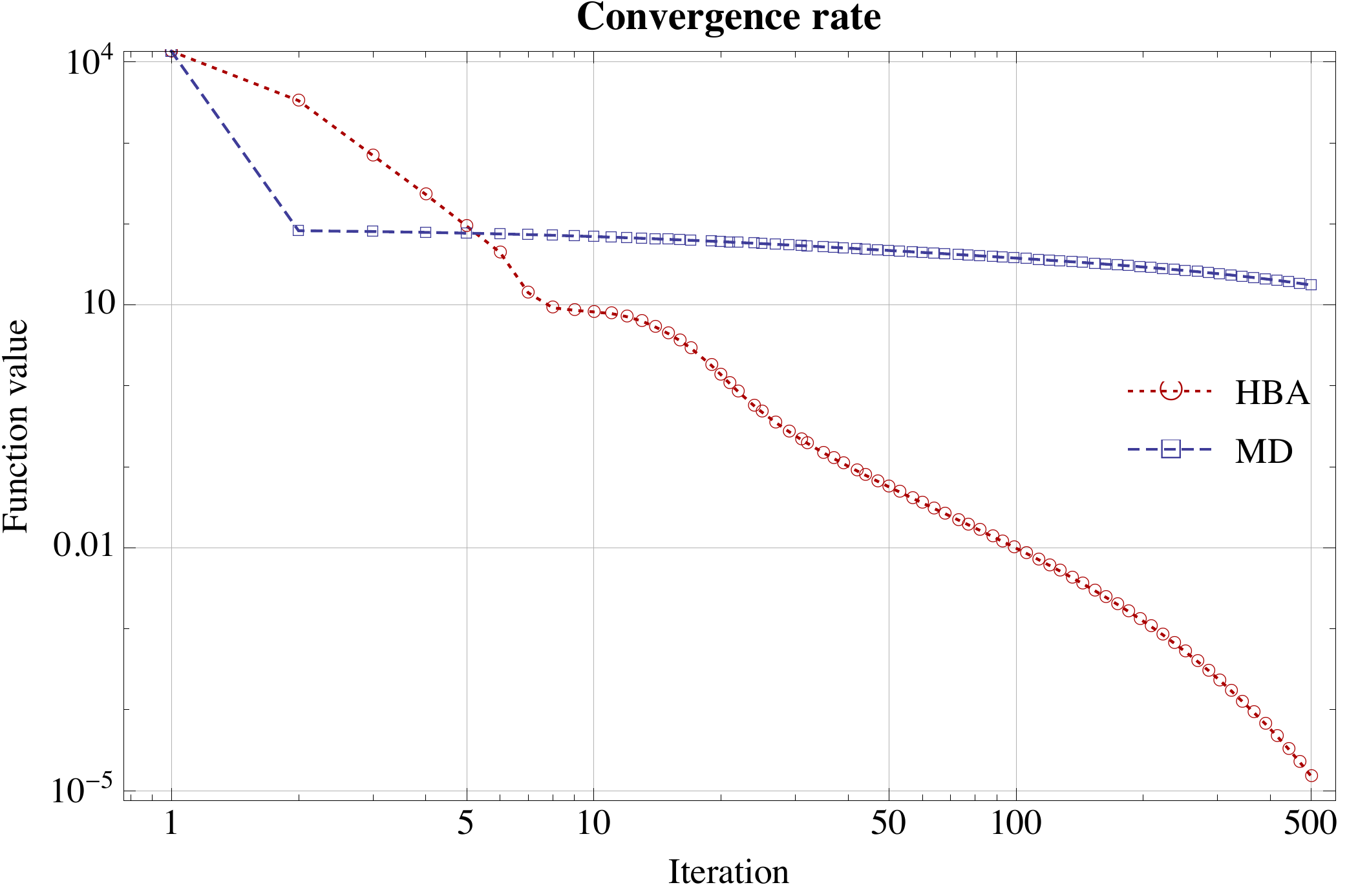}
\caption{Convergence rate for the Rosenbrock and Beale functions (left and right respectively).}
\label{fig:test-rate}
\end{subfigure}
\caption{Convergence of \eqref{eq:HBA} in the case of the Rosenbrock and Beale test functions (\cref{eq:Rosenbrock,eq:Beale} respectively).
The convergence rate of \eqref{eq:HBA} is compared to that of a standard \acl{MD} algorithm;
all experiments were run with the entropic kernel $\theta(x) = x \log x$.}
\label{fig:test}%
\end{figure}


As a first illustration of the convergence of \eqref{eq:HBA}, we focus on two low-dimensional test functions that are widely used in the global optimization literature:
\begin{enumerate}
\addtolength{\itemsep}{\medskipamount}
\item
The Rosenbrock function:
\begin{flalign}
\obj(x_{1},x_{2})
	&= 100 (x_{2} - x_{1^{2}})^{2} + (1- x_{1})^{2},
\label{eq:Rosenbrock}
	\\
\intertext{with input domain $x_{1}, x_{2} \in [-3,3]$.
\item
The Beale function:}
\obj(x_{1},x_{2})
	&= (1.5 - x_{1} + x_{1}x_{2})^{2}
	\notag\\
	&+ (2.25 - x_{1} + x_{1}x_{2}^{2})^{2}
	+ (2.625 - x_{1} + x_{1}x_{2}^{3})^{2},
\label{eq:Beale}
\end{flalign}
with input domain $x_{1}, x_{2} \in [-4,4]$.
\end{enumerate}
\smallskip
The Rosenbrock function is a non-convex unimodal function with a unique global minimum located at the lowest point of a very flat and thin parabolic valley which is notoriously difficult for first-order methods to traverse.
The Beale function is a non-convex multimodal function with very sharp peaks at the corners of the input domain which cause considerable difficulties to aggressive step-size policies.

In \cref{fig:test}, we plot two test runs of the \acl{HBA} (\cref{alg:HBA}) with the negative entropy kernel $\theta(x) = x \log x$ and a random initialization.
For benchmarking purposes, we also ran the corresponding \acl{MD} algorithm \eqref{eq:MD} with the same initialization, step-size and kernel function.
The sample \ac{HBA} trajectories are shown in \cref{fig:test-trajectories} and are seen to converge to a solution of \eqref{eq:Opt}.
Subsequently, the value convergence rate of the algorithm is plotted in \cref{fig:test-rate}:
the log-log scale of the plot indicates a monotonic decrease following a power law convergence rate, consistent with the theoretical predictions of \cref{thm:rate} (the non-uniformity of the algorithm's speed has to do with the very flat valleys/plateaus that the algorithm needs to traverse in order to approach a solution).

\revise{Even though we do not report the results here, a similar behavior was observed in all the common benchmarks (Himmelblau, Styblinski-Tang, etc.) and kernels (Burg, Hellinger, etc.) that we tested.
We find this feature of the \acl{HBA} particularly appealing for practical applications, especially for objectives with a complex landscape.}

\subsection{Applications to traffic routing}
\label{sec:routing}

As a concrete application of our results, we focus below on the \acdef{TAP}, a key problem in transportation and network science that concerns the optimal selection of paths between origins and destinations in traffic networks.
Referring to \cite{BG92,NRTV07} for a detailed discussion, the main ingredients of the problem are as follows:
First, let $\graph = (\vertices,\edges)$ be a directed multi-graph with vertex set $\vertices$ and edge set $\edges$.
Assume further that there is a finite set of \acdef{OD} pairs indexed by $\pair\in\pairs$, each with an individual \emph{traffic demand} $\rate^{\pair}\geq 0$ that is to be routed from the pair's \emph{origin node} $\source^{\pair}\in\vertices$ to its \emph{destination} $\sink^{\pair}\in\vertices$.
To route this traffic, the $\pair$-th \ac{OD} pair employs a set $\routes^{\pair}$ of \emph{paths} joining $\source^{\pair}$ to $\sink^{\pair}$, with each path $\route\in\routes^{\pair}$ comprising a sequence of edges that meet head-to-tail in the usual way.%
\footnote{Specifically, we do not assume that $\routes^{\pair}$ is necessarily the set of \emph{all} paths joining $\source^{\pair}$ to $\sink^{\pair}$, but only some subset thereof.
This distincition is important in packet-switched networks where, typically, only a set of paths with minimal hop count are used for traffic routing.}

Now, writing $\routes \equiv \bigcup_{\pair\in\pairs} \routes^{\pair}$ for the ensemble of all such paths, the set of feasible \emph{routing flows} $\flow = (\flow_{\route})_{\route\in\routes}$ in the network is defined as
\begin{equation}
\label{eq:flows}
\textstyle
\flows
	= \setdef*{\flow\in\R_{+}^{\routes}}{\text{$\sum_{\route\in\routes^{\pair}} \flow_{\route} = \rate^{\pair}$ for all $\pair\in\pairs$}}.
\end{equation}
In turn, a routing flow $\flow\in\flows$ induces a \emph{load} on each edge $\edge\in\edges$ as
\begin{equation}
\label{eq:load}
\load_{\edge}
	= \sum_{\route\ni\edge} \flow_{\route},
\end{equation}
and we write $\load = (\load_{\edge})_{\edge\in\edges}$ for the corresponding \emph{load profile} on the network. 
Given all this, the delay (or latency) experienced by an infinitesimal traffic element traversing edge $\edge$ is determined by a nondecreasing continuous \emph{cost function} $\cost_{\edge}\from[0,\infty)\to[0,\infty)$:
more precisely, if $\load = (\load_{\edge})_{\edge\in\edges}$ is the load profile induced by a feasible routing flow $\flow = (\flow_{\route})_{\route\in\routes}$, the incurred delay on edge $\edge\in\edges$ is $\cost_{\edge}(\load_{\edge})$.
Hence, with a slight abuse of notation, the associated cost of path $\route\in\routes$ will be
\begin{equation}
\label{eq:cost-path}
\cost_{\route}(\flow)
	= \sum_{\edge\in\route} \cost_{\edge}(\load_{\edge}),
\end{equation}
and the aggregate latency in the network will be given by
\begin{equation}
\label{eq:cost-total}
\Cost(\flow)
	= \sum_{\pair\in\pairs} \sum_{\route^{\pair}\in\routes^{\pair}} \cost_{\route^{\pair}}(\flow)
	= \sum_{\route\in\routes} \cost_{\route}(\flow).
\end{equation}
Accordingly, with all this at hand, the goal of the \acl{TAP} is to identify a flow profile that minimizes the aggregate latency in the network, i.e., solve the continuous, nonlinear problem
\begin{equation}
\label{eq:TAP}
\tag{TAP}
\begin{aligned}
\textrm{minimize}
	&\quad
	\Cost(\flow)
	\\
\textrm{subject to}
	&\quad
	\flow\in\flows.
\end{aligned}
\end{equation}

Since \eqref{eq:TAP} is a linearly constrained problem, the proposed \ac{HBA} algorithm can be applied essentially ``off the shelf''.
To do so, we consider an experimental setup consisting of a Barabasi-Albert random graph with $\abs{\vertices} = 50$ nodes and $\nPairs$ \acl{OD} pairs chosen uniformly at random from the generated graph.
Subsequently, we used a variant of Dijkstra's algorithm to pick out $\abs{\routes^{\pair}} = 20$ minimal hop count paths per \ac{OD} pair, and we drew the corresponding traffic demands $\rate^{\pair}$, $\pair\in\pairs$, uniformly at random from $[0,1]$.
Concretely, in our experiments, we took $\nPairs=100$ and $\nPairs=500$, implying in turn that the dimensionality $n = \sum_{\pair\in\pairs} \abs{\routes^{\pair}}$ of the resulting \acl{TAP} is $n=1000$ or $n=2500$ respectively.
The network's edge cost functions were also drawn randomly following a straightforward linear model of the form $\cost_{\edge}(\load) = a_{\edge} + b_{\edge} \load_{\edge}$, with $a_{\edge}$ and $b_{\edge}$ drawn uniformly at random from $[0,10]$ and $[0,1]$ respectively.


\begin{figure}[tbp]
\includegraphics[width=.49\textwidth]{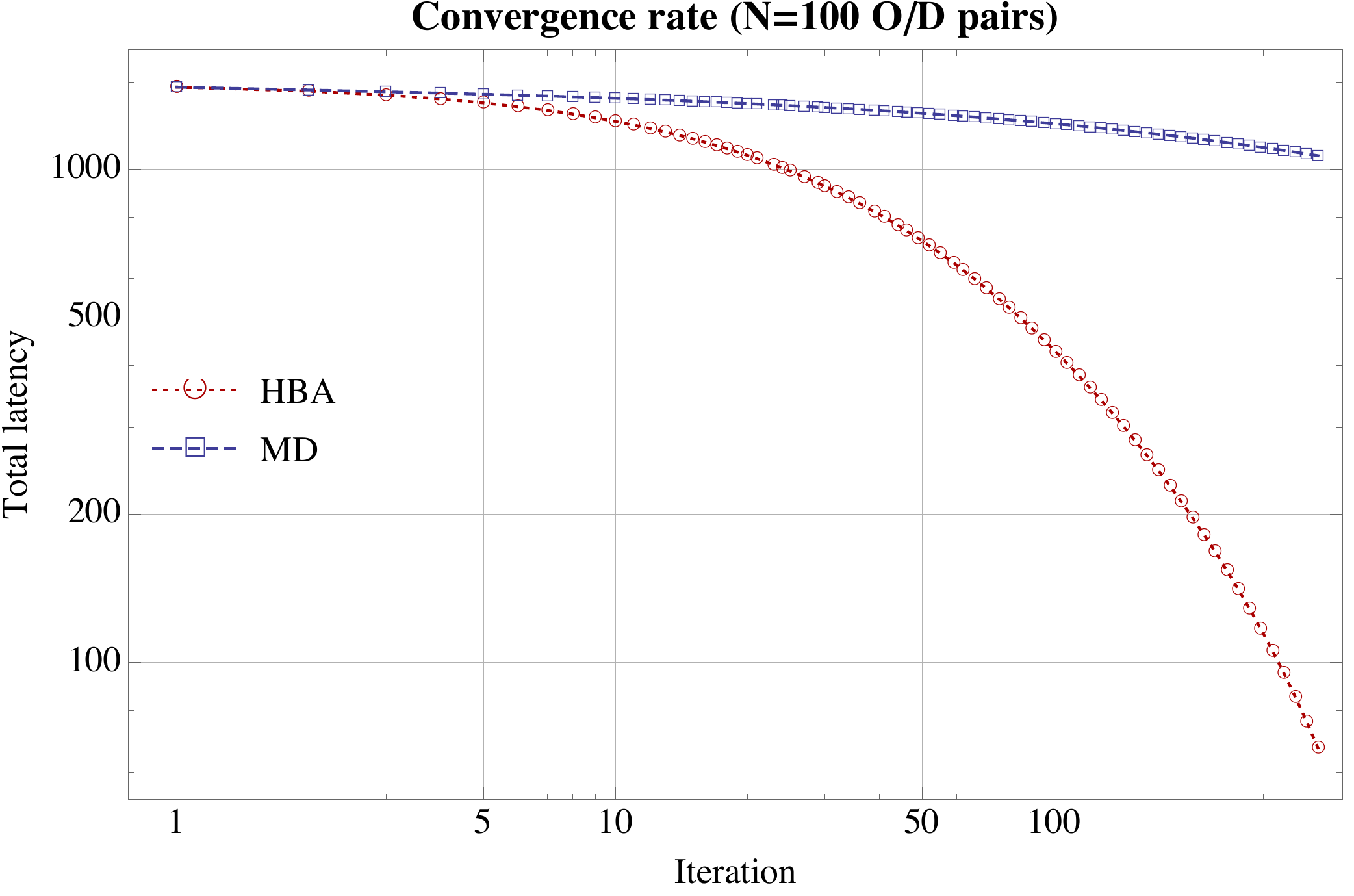}
\hfill
\includegraphics[width=.49\textwidth]{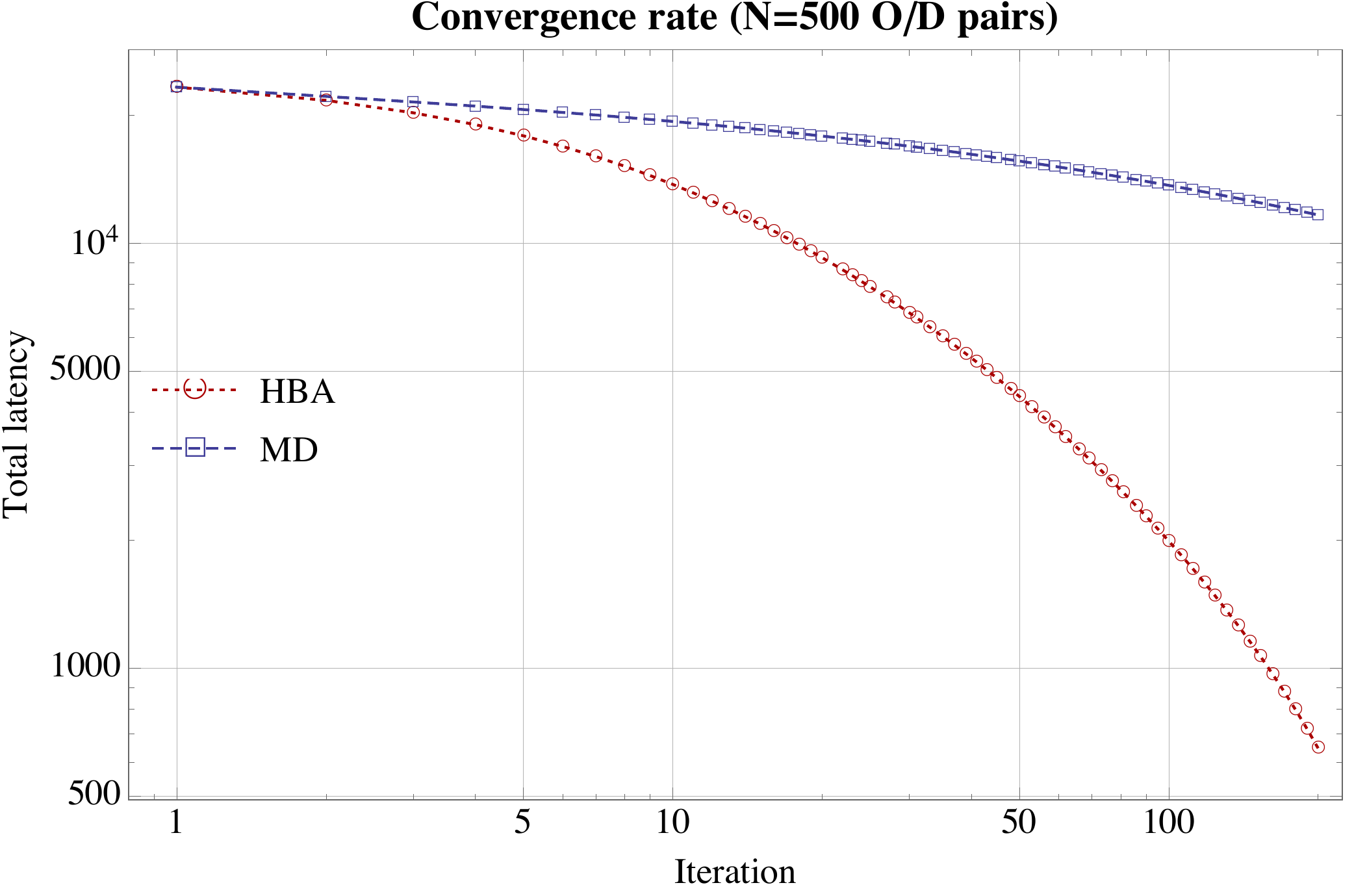}
\caption{Convergence of \cref{alg:HBA} in the \acl{TAP} \eqref{eq:TAP}.
The base network is a randomly drawn Barabasi-Albert graph with $\abs{\vertices} = 50$ nodes and $\nPairs=100$ or $\nPairs=500$ \ac{OD} pairs (left and right respectively).
In both cases, \cref{alg:HBA} exhibits a very fast rate of convergence relative to standard \acl{MD} methods.}
\label{fig:routing}%
\end{figure}


\revise{Our results are shown in \cref{fig:routing}.
In detail, since the problem's feasible region is a high-dimensional simplex (or, rather, a product thereof), we focused on the negative entropy kernel $\theta(\point) = \point\log\point$ which is known to achieve a (nearly) dimension-free convergence rate for \acl{MD} \cite{NY83,BecTeb03}.
Subsequently, we ran both \acl{HBA} and \acl{MD} with the uniform traffic assignment initialization $\flow_{\route}^{\pair} = \rate^{\pair}/\abs{\routes^{\pair}}$, $\route\in\routes^{\pair}$, $\pair\in\nPairs$, which is standard in the traffic assignment literature \cite{BG92,VPM19}.}
In both cases, the \ac{HBA} algorithm exhibits great gains in total latency after no more than a few hundred iterations:
specifically, we observe a total latency reduction of over $95\%$ relative to uniform traffic assignment, and over $90\%$ relative to \acl{MD} after the same number of iterations.
Given the problem's dimensionality of a few thousand control variables, this represents a gain that is particularly encouraging for other applications of the algorithm to large-scale optimization problems.

\section{Conclusion}
\label{sec:conclusion}

In this paper, we presented a class of first-order methods that includes as special cases several widely used numerical schemes for solving (possibly non-convex) smooth optimization problems with linear constraints.
Motivated by the continuous-time \acl{HR} gradient dynamics of \cite{ABB04}, we construct \typo{a} computationally efficient algorithm which avoids the need for a prox-step.
We call this method the \acdef{HBA}.
We show that \ac{HBA}, accompanied with a line search procedure based on Armijo backtracking, yields convergence to \ac{KKT} points.
In case of quadratic programming,
we also provide a sublinear value convergence rate.
Interestingly, the rate depends on the employed metric, highlighting its importance as a design choice.

There are several interesting and challenging open questions left for future research.
A first step concerns the extension of \ac{HBA} methods to non-smooth problems:
in particular, the key driver in proving global convergence is the lower bound on the algorithm's step-size sequence.
From the proof of \cref{lem:lower}, it is clear that we can actually weaken the smoothness assumption made on the objective function significantly in that regard.
We therefore conjecture that it is possible to extend our arguments to problems in which the objective function $\obj$ is not smooth, which would allow us to apply \eqref{eq:HBA} to important applications in statistics and signal processing \cite{HaeLiuYe18}.

\revise{To better assess the method's total oracle complexity, it is important to make a distinction between gradient and function evaluations.
With regard to the former, a key extension of our work would be to an accelerated version of \eqref{eq:HBA}:
recently, \cite{GhaLan16} introduced a gradient method for non-convex optimization problems,
raising the question whether this method can be combined with \acl{HR} gradient steps.
On the other hand, to estimate the number of function evaluations per iteration / gradient call, one would need to establish a bound on the number of Armijo backtracking steps per iteration.
Given the highly nonlinear dependence of the bootstrap step-size $\step_{0}(\point)$ on the problem's primitives, this question seems to be a fairly challenging technical exercise which we leave for future work.}

Finally, we should mention that we have presented \eqref{eq:HBA} as a generic template for first-order methods:
the search direction $v(x)$ can be changed to other data structures, such as a statistical estimator for the gradient, or the profile of individual gradients in a game-theoretic problem.
This opens the door to analyze \eqref{eq:HBA} in the context of stochastic optimization and/or variational inequalities.
This would provide a unifying framework for the recent results of \cite{HaeLiuYe18,LiLiuYaoYe17};
we delegate this technically challenging question to future work.

\section*{Acknowledgments}
%
%
\typo{The authors are indebted to the associate editor and two anonymous referees for their detailed suggestions and remarks.
Mathias Staudigl would also like to thank the University of Vienna for its hospitality while finishing this paper.}
%
%
Panayotis Mertikopoulos was partially supported by the French National Research Agency (ANR) project ORACLESS (ANR-16-CE33-0004-01).
Mathias Staudigl and Panayotis Mertikopoulos were partially supported by the COST Action CA16228 ``European Network for Game Theory'' (GAMENET).




\bibliographystyle{siam}
\bibliography{../bibtex/IEEEabrv,../bibtex/Bibliography-SMD}
\end{document}